\documentclass[11pt]{article}
\usepackage[utf8]{inputenc}
\usepackage[a4paper,top=2.5cm,bottom=2cm,left=2cm,right=2cm,marginparwidth=2cm]{geometry}
\usepackage{amsmath}
\usepackage{amssymb}
\numberwithin{equation}{section}
\usepackage[english]{babel}
\usepackage{amsthm}
\usepackage{bbm}
\usepackage{amsfonts}
\usepackage{comment}
\usepackage{mathrsfs}
\usepackage{hyperref}
\usepackage{titlesec}

\usepackage{todonotes}

\hypersetup{
    colorlinks=true,
    linkcolor=blue,
    citecolor=red,
    filecolor=magenta,      
    urlcolor=black,
    pdftitle={Topologically protected modes in dispersive materials: the case of undamped systems},
    pdfpagemode=FullScreen,
    }
\usepackage{graphicx, tikz}
\usepackage{float}
\usepackage{xcolor}
\usepackage{bbm}
\usepackage[format=plain,
            font=it]{caption}

\newtheorem{theorem}{Theorem}[section]
\newtheorem{corollary}{Corollary}[theorem]
\newtheorem{lemma}[theorem]{Lemma}

\newtheorem*{remark}{Remark}

\newtheorem{definition}[theorem]{Definition}

\newcommand{\upd}{\mathrm{d}}  

\def\a{\alpha }     \def\b{\beta  }         
     \def\d{\delta}          
    \def\ve{\varepsilon}    \def\z{\zeta }
       \def\th{\theta }        
\def\k{\kappa }          \def\l{\lambda }    
\def\L{\Lambda }    \def\m{\mu}                      
                \def\s{\sigma }     
                \def\u{\upsilon }   
   \def\f{\phi }                
                         
       \def\w{\omega }

\title{Topologically protected modes in dispersive materials: the case of undamped systems}
\author{Konstantinos Alexopoulos\thanks{Department of Mathematics, ETH Zurich, R\"amistrasse 101, CH-8092 Zurich, Switzerland.} \and Bryn Davies\thanks{Department of Mathematics, Imperial College London, 180 Queen's Gate, London SW7 2AZ, United Kingdom.}}
\date{}

\begin{document}

\maketitle

\begin{abstract}
    This work extends the theory of topological protection to dispersive systems. This theory has emerged from the field of topological insulators and has been established for continuum models in both classical and quantum settings. It predicts the existence of localised interface modes based on associated topological indices and shows that, when such modes exist, they benefit from enhanced robustness with respect to imperfections. This makes topologically protected modes an ideal starting point for building wave guiding devices. However, in many practical applications such as optics or locally resonant meta-structures, materials are dispersive in the operating frequency range. In this case, the associated spectral theory is less straightforward. This work shows that the existing theory of topological protection can be extended to dispersive settings. We consider time-harmonic waves in one-dimensional systems with no damping.

\end{abstract}

\tableofcontents

\section{Introduction}

The existence of strongly localised existence eigenmodes in perturbed periodic media is one of the foundations of modern wave physics. Such modes allow for waves of specific frequencies to be strongly localised at desired locations and are the starting point for the design of many wave guiding and control devices.  

The pre-eiminent theory in the field of periodic waveguides is the notion of topological protection \cite{khanikaev2013photonic}. This field was inspired by ideas related to the famous quantum hall effect \cite{klitzing1980QHE} and was successfully translated to classical wave systems around 15 years ago \cite{wang2009observation}. This theory uses topological indices associated to the underlying periodic structures \cite{asboth2016short} to predict the existence of localised edge modes at interfaces and edges of the material. Further, the invariant nature of topological indices means we can expect that such modes experience enhanced robustness with respect to imperfections. Given how desirable a priori robustness properties are for being able to confidently manufacture functional devices, topological waveguides has emerged as an important sub-field of wave physics \cite{PriceRoadmap}.

Alongside the rapid developments in the physical literature, a corresponding mathematical theory for topologically protected modes has been developed. This theory first emerged in the setting of the Schr\"odinger equation \cite{fefferman2014topologically, fefferman2017topologically}, but has since been extended to classical wave systems \cite{lin2022mathematical, coutant2023surface}. Many of the seminal mathematical studies consider one-dimensional systems, since the fundamental mechanism requires an interface or edge to be introduced in just one axis of periodicity. However, multi-dimensional systems can be studied by using \emph{e.g.} integral operators \cite{ammari2022robust, ammari2020topologically} or a reduction to Dirac operators \cite{fefferman2014wave, drouot2019bulk, ammari2020high}.

The existing mathematical theory mostly studies systems with frequency-independent material parameters. This is not only a natural toy model but gives a good description of several physical settings (\emph{e.g.} microwaves and perfect conductors). However, many important materials are dispersive in desirable frequency ranges. This includes locally resonant systems (\emph{e.g.} coupled Helmholtz resonators \cite{zhao2021subwavelength}) and most metals at optical frequencies. This work extends the mathematical theory of topologically protected modes to dispersive settings. The main theoretical challenge for doing so is that the spectrum of the differential operator can no longer be understood using standard linear eigenvalue theory. Instead, the spectrum is obtained through the solution of a non-linear eigenvalue problem. Further, many important dispersive materials have poles and singularities at certain frequencies \cite{alexopoulos2023effect, touboul2023high}.

In this work, we will focus on materials whose permittivity is dispersive and real-valued (meaning there is no damping in the system). We exploit the work of \cite{alexopoulos2023effect} which characterised the Bloch spectrum as a function of the singularities of canonical dispersive permittivities. The main idea of this work is to combine this theory with the work of \cite{coutant2023surface} (and the previous works upon which it builds, \emph{e.g.} \cite{fefferman2014topologically, lin2022mathematical}), which uses the surface impedance to characterise the existence and topological protection of interface modes.

In Section \ref{section:MS}, we establish the mathematical setting describing the problem and the geometry governing the system. In Section \ref{section:IME}, we define the concept of the impedance function and, exploiting the properties of the mirror symmetry of the periodically repeated cells, we generalise the monotonicity argument used in \cite{coutant2023surface} and we show that in each band-gap, there exists a unique non-zero frequency for which an interface mode exists. Then, in Section \ref{section:ZP}, we define the Zak phase and we show why it is an invariant property of the system. In Section \ref{section:AB}, using the transfer matrix method, we show the existence of an asymptotically decaying mode at infinity and we provide the condition that needs to be satisfied for that mode to unique. Finally, in Section \ref{section:TP}, we study how this interface mode is affected by perturbations in the material parameters or the symmetry of the system.

\section{Mathematical setting}\label{section:MS}

Let us first present the mathematical setting in which we will work. We will consider materials composed of periodically repeating unit cells, within which the permittivity is allowed to vary as a function of both frequency and position. A crucial assumption will be that each unit cell is mirror symmetric. That is, the material's permittivity is described by a function $\epsilon(x,\omega)$ which is periodic ($\epsilon(x,\cdot)=\epsilon(x+1,\cdot)$) and symmetric ($\epsilon(x+h,\cdot)=\epsilon(x+1-h,\cdot)$) in its first variable. Further, our analysis will require an assumption that $\epsilon(\cdot,\omega)$ is a piecewise non-decreasing function of $\omega$, in the sense that $\frac{\upd\ve}{\upd\w}\geq0$ whenever the derivative exists.

In what follows, in order to ease the notation, our setting will be constituted by frequency dispersive permittivities, piecewise constant with respect to the space variable. Although, we note that our analysis is the same and the results still hold for permittivities which respect the mirror symmetry as mentioned above.

We consider two dispersive materials $A$ and $B$. Each one is of the form of a semi-infinite array. The arrays are glued together at $x_0=0$ and we assume that the material $A$ expands towards $-\infty$ and the material $B$ expands towards $+\infty$. In addition, we assume that each material is constructed by repeating periodically a unit cell. Each unit cell is the product of layering, i.e. it has particles of two different permittivities $\ve_1$ and $\ve_2$. A key assumption is that each permittivity is a dispersive, continuous function of the frequency $\w\in\mathbb{R}$ of the system. Let us denote by $D^{[1]}_i$, $i=1,2,\dots,N+1,$ and by $D^{[2]}_j$, $j=1,2,\dots,N,$ the particles of the unit cell with permittivity $\ve_1$ and $\ve_2$, respectively. Then, the permittivity of the system is defined as follows:
\begin{align}\label{def:permittivity}
    \ve(x,\w) = 
    \begin{cases}
        \ve_1(\w), & x \in D^{[1]} := \bigcup_{i=1}^{N+1} D^{[1]}_i,\\
        \ve_2(\w), & x \in D^{[2]} := \bigcup_{i=1}^N D^{[2]}_i.
    \end{cases} 
\end{align}

We define the sequence $\{x_n\}_{n\in\mathbb{Z}}$ to be the set of endpoints of each one of the periodically repeated cells. We take $x_0$ to be the glue point of the two materials and so for $n>0$ we are in material $B$ and for $n<0$, we are in material $A$.

In Figure \ref{fig:materials_AandB} we provide a schematic depiction of such an example of periodic cells for the materials $A$ and $B$, respectively. Concerning the construction characteristics of the periodic cell, for materials $A$ and $B$, we have $3$ particles of permittivity $\ve_1$, i.e. $D^{[1]}_1,D^{[1]}_2$ and $D^{[1]}_3$, and $2$ particles of permittivity $\ve_2$, i.e. $D^{[2]}_1$ and $D^{[2]}_2$. We note that our result can be generalised in settings where the number of particles in the periodic cells is greater. We also notice the geometry of each one of the periodic cells, i.e. the symmetric way in which the particles are placed in each periodic cell.

\begin{figure}
    \centering
    \begin{tikzpicture}
        \node[inner sep=0pt] (russell) at (-4.5,6.5) {\includegraphics[width=0.45\textwidth]{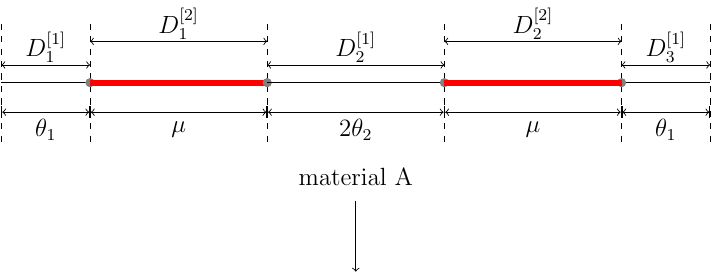}};
        \node[inner sep=0pt] (russell) at (4.5,6.5) {\includegraphics[width=0.45\textwidth]{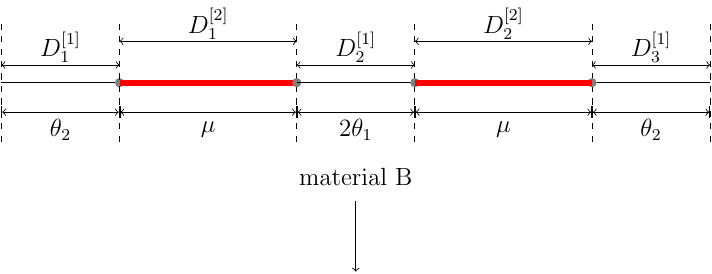}};
        \node[inner sep=0pt] (russell) at (0,4) {\includegraphics[width=1.05\textwidth]{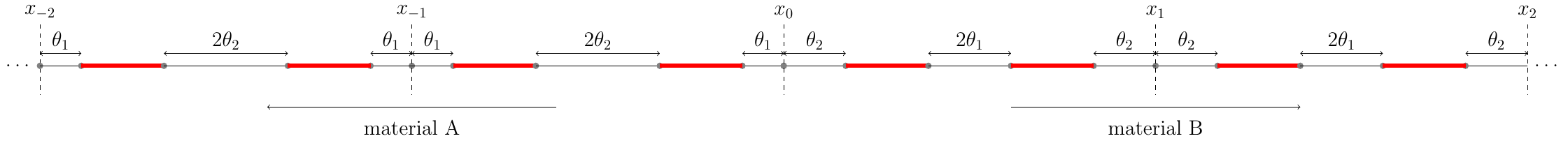}};
    \end{tikzpicture}
    \caption{An example of the setting studied. In the unit cell of each material we have 3 particles of permittivity $\ve_1$ and 2 particles of permittivity $\ve_2$. We notice the mirror symmetric way in which the particles are placed inside the periodic cells. Each material is constituted by a semi-infinite array created by periodically repeating the unit cell. Our structure is the result of gluing materials $A$ and $B$ at the point $x_0$.} 
    \label{fig:materials_AandB}
\end{figure}

We are interested in finding eigenvalues 
\begin{align}\label{eq:Helmholtz}
    \mathcal{L}u=\omega^2u
\end{align}
of the ordinary differential operator
\begin{equation}\label{def:L}
    \mathcal{L}u := - \frac{1}{\mu_0}\frac{\upd }{\upd x} \left( \frac{1}{\ve(x,\w)} \frac{\upd u}{\upd x} \right).
\end{equation}
In particular, the eigenmodes of interest are those which decay as $|x|\rightarrow\infty$, such that they are localised in a neighbourhood of the interface at $x_0=0$. 

To find candidate eigenvalues $\omega^2$ for which localised eigenmodes can occur, it is valuable to consider the Floquet-Bloch spectrum of the operators associated to materials A and B. This is the set of eigenmodes which satisfy the Floquet-Bloch boundary conditions
\begin{align}\label{bd1}
    u(x+1) = e^{i\k}u(x)\quad\text{and}\quad \frac{\upd u}{\upd x}(x+1) = e^{i\k}\frac{\upd u}{\upd x}(x),
\end{align}
for some $\k\in\mathcal{B}:=[-\pi,\pi]$. Hence, these eigenmodes belong to the space of functions
\begin{align*}
    L^2_{\k} := \Big\{ f\in L^2_{\mathrm{loc}}: \ f(x+1) = e^{i\k} f(x), \ \k\in\mathbb{R} \Big\}.
\end{align*}

It follows that the localised eigenmodes that we are looking for, have to be in band gaps of the two materials. This happens because we are looking for solutions to \eqref{def:L}, which decay as $|x|\to+\infty$. Thus, if the modes are in the bands, then from \eqref{bd1}, we see that their magnitude is preserved. Let us also note here that we are interested in materials $A$ and $B$ which have overlapping band gaps, otherwise such modes cannot exist in this setting.

On its own, studying the spectrum of the operator $\mathcal{L}$ is a complicated task. This is due to the fact that there is a non-linear dependence of the operator on the eigenvalues $\w^2$ via $\ve(x,\w)$. In addition, the permittivity function $\ve$ may present certain aspect, e.g. poles, which make this analysis even harder. In \cite{alexopoulos2023effect}, we have studied how these properties of the permittivity function affect the structure of the spectrum of the differential operator $\mathcal{L}$. In general, its spectrum will consist of countably many spectral bands. Unusual behaviour can occur near to the poles of the singularity, however, where there can be a countably infinite number of bands clustered within a finite region of the pole.

We will argue in the same way as in \cite{coutant2023surface}. We will generalise their results to our dispersive setting (which differs from their dispersive example as the dispersion appears in the electric permittivity, rather than the magnetic permeability, in \eqref{def:L}). We will include proofs of most statements for completeness, although some are straighforward generalisations of \cite{coutant2023surface}. We also note that in the case where the permittivity $\ve$ no longer depends on the frequency $\w$ but only on $x$, then we are in a similar setting as in \cite{lin2022mathematical}.

\section{Interface mode existence}\label{section:IME}


The main theoretical challenge is to show that an interface mode exists for the structure composed of the two materials.


\subsection{Impedance functions}

Let us see a sufficient condition that needs to be satisfied for the interface mode to exist. We introduce the following notation.

\begin{definition}
    Let $f:\mathbb{R}\to\mathbb{C}$ be a continuous function and let $\a\in\mathbb{R}$. Then, we define
    \begin{align*}
        f(\a^{-}) := \lim_{x\uparrow\a} f(x) \ \ \ \text{ and } \ \ \ f(\a^{+}) := \lim_{x\downarrow\a} f(x).
    \end{align*}
\end{definition}

The condition that needs to be satisfied for the interface mode to exist is continuity of both the solution and its derivative at the interface. Using the above notation, the continuity conditions that needs to be satisfied at the interface $x_0$ can be written as
\begin{align}\label{eq:continuityconditions}
    u(x_0^{-}) = u(x_0^{+}) \ \ \ \text{ and } \ \ \ \left(\frac{1}{\ve_1(\w)}\frac{\upd u}{\upd x}\right)(x_0^{-}) =  \left(\frac{1}{\ve_1(\w)}\frac{\upd u}{\upd x}\right)(x_0^{+}).
\end{align}

We can now give the definition for the impedance functions.

\begin{definition}[Surface impedances]
    We define the surface impedances of the problem \eqref{eq:Helmholtz}, as follows:
    \begin{align}
        Z^-(\w)= -\frac{u(x_0^{-})}{\frac{1}{\ve_1(\w)}\frac{\upd}{\upd x}u(x_0^{-})} \ \ \ \text{ and } \ \ \ Z^+(\w)= \frac{u(x_0^{+})}{\frac{1}{\ve_1(\w)}\frac{\upd}{\upd x}u(x_0^{+})}, \ \ \ \w\in\mathbb{R}.
    \end{align}
\end{definition}

Then, the following result is a direct consequence of the continuity conditions \eqref{eq:continuityconditions}.

\begin{lemma}\label{lemma:impedance}
    Let $\mathfrak{A}\subset\mathbb{R}$. Then, a necessary and sufficient conditions for the existence of an interface localised mode in $\mathfrak{A}$ is
    \begin{align}\label{eq:impedancecondition}
        Z^+(\w) + Z^-(\w) = 0, \ \ \ \w\in\mathfrak{A}.
    \end{align}
\end{lemma}

The question that rises naturally is whether a frequency $\w\in\mathbb{R}$ exists such that this condition is satisfied and so, the global mode exists.

\subsection{Dispersion relation}

The quasiperiodic Helmholtz problem \eqref{def:L}-\eqref{bd1} is characterised by a dispersion relation. This is an expression which relates the quasiperiodicity $\k$ with the frequency $\w$, such that it describes the spectral bands $\omega=\omega_n(\k)$. It is an equation of the form:
\begin{align}\label{eq:disp_rel}
    2\cos(\k) = f(\w),
\end{align}
where $f:\mathbb{R}\rightarrow\mathbb{R}$ is a function that depends on the material parameters and the system's geometry. This function is variously known as the discriminant or Lyapunov function of the operator $\L$ \cite{kuchment2016overview}. Some examples are provided in \emph{e.g.} \cite{alexopoulos2023effect, morini2018waves}.

The analysis that follows stems from understanding the symmetries of the Bloch modes at the edged of the band gaps. Hence, it will be helpful to know that the maxima and minima of any band must occur at $\kappa=0$ or $\kappa=\pm\pi$. This follows from the fact that the spectral bands are monotonic functions of $\k$ within the reduced Brillouin zone $[0,\pi]$.

\begin{lemma}[Frequency monotonicity]
    It holds that the frequency $\w\in\mathbb{R}$ satisfying the \eqref{eq:disp_rel}, as a function of the quasiperiodicity $\k$, is a monotonic function in the reduced Brillouin zone $[0,\pi]$.
\end{lemma}

\begin{proof}
    Indeed, we differentiate \eqref{eq:disp_rel} with respect to the quasiperiodicity $\k$ and we get
    \begin{align*}
        -2\sin(\k) = \frac{\upd f}{\upd\w}\frac{\upd\w}{\upd\k}.
    \end{align*}
    We know that in $[0,\pi]$, it holds $-2\sin(\k)\leq0$ with the zeros occurring in $\{0,\pi\}$. It is a direct consequence of this that $\frac{\upd\w}{\upd\k}$ keeps a constant sign in $(0,\pi)$. Also, if $\frac{\upd\w}{\upd\k}(\k)=0$, then $\k\in\{0,\pi\}$. Thus, $\w(\k)$ is monotonic in the reduced Brillouin zone. This concludes the proof.
\end{proof}

\begin{remark}
    We note that the above property holds for any sufficiently smooth function $f:\mathbb{R}\to\mathbb{R}$.
\end{remark}

Thus, since $\w$ is a monotonic function of $\k$ in the reduced Brillouin zone, it follows that the edges of the spectrum are at $0$ and $\pi$.

\subsection{Mirror symmetry}

Let us now focus on the geometry inside the periodic cells. The first thing we observe is the mirror symmetry which is established in the periodic cells of both materials. Indeed, we observe that for $x\in[0,1]$, we have
\begin{align}
    \ve(x,\w) = \ve(1-x,\w).
\end{align}
We introduce the parity operator $\mathcal{P}$, which acts on a function $f$ as follows:
\begin{align*}
    (\mathcal{P}f)(x) := f(1-x).
\end{align*}
Then, the mirror symmetry implies
\begin{align}\label{eq:paritysymmetry}
    \mathcal{P}\mathcal{L}(-\k,\w) = \mathcal{L}(\k,\w) \mathcal{P}, \ \ \ \forall \k\in\mathcal{B}.
\end{align}
As a direct consequence of \eqref{eq:paritysymmetry}, we have the following lemmas.

\begin{lemma}
    Let $\w\in\mathbb{R}$. Then:
    \begin{itemize}
        \item $\mathcal{L}(-\pi,\w)=\mathcal{L}(\pi,\w)$.
        \item The differential operator $\mathcal{L}$ and the parity operator $\mathcal{P}$ commute at $\k\in\{0,\pm\pi\}$.
    \end{itemize}
\end{lemma}

\begin{proof}
    We divide this proof in two parts, one for each result.
    \begin{itemize}
    
        \item We know that $\w_n(-\k)=\w_n(\k)$. Then, from \eqref{def:L}, it follows directly that $\mathcal{L}(-\pi,\w)=\mathcal{L}(\pi,\w)$.
        
        \item From \eqref{eq:paritysymmetry}, we get
        \begin{align*}
            \mathcal{P}\mathcal{L}(0,\w) = \mathcal{L}(0,\w) \mathcal{P}.
        \end{align*}
        We also get
        \begin{align*}
            \mathcal{P}\mathcal{L}(-\pi,\w) = \mathcal{L}(\pi,\w) \mathcal{P}.
        \end{align*}
        We have shown that $\mathcal{L}(-\pi,\w)=\mathcal{L}(\pi,\w)$. This gives
        \begin{align*}
            \mathcal{P}\mathcal{L}(\pi,\w) = \mathcal{L}(\pi,\w) \mathcal{P} \ \ \ \text{ and } \ \ \ \mathcal{P}\mathcal{L}(-\pi,\w) = \mathcal{L}(-\pi,\w) \mathcal{P},
        \end{align*}
        which is the desired result.

    \end{itemize}
    This concludes the proof.
\end{proof}

In order to show the dependence of a solution $u$ on the spatial variable $x$ but also on both the quasiperiodicity $\k$ and the frequency $\w$, we will use the notation $u^{[\k]}(x,\w)$.

\begin{lemma}
    Let $\l_n(\w_n):=\w^2_n(\k)$ be an eigenvalue of $\mathcal{L}(\k,\w_n)$, with $u_n^{[\k]}(x,\w_n)$ being the associated eigenvector. Then:
    \begin{itemize}
        \item $\mathcal{P}u_n^{[\k]}(x,\w_n)$ is an eigenvector $\mathcal{L}(-\k,\w_n)$ with the same eigenvalue $\l_n(\k)$.
        \item $u_n^{[-\k]}(x,\w_n)$ is also an eigenvector of $\mathcal{L}(-\k,\w_n)$ with eigenvalue $\l_n(\k)$.
    \end{itemize}
\end{lemma}

\begin{proof}
    We divide this proof in two parts, one for each result.
    \begin{itemize}
    
        \item We observe that \eqref{eq:paritysymmetry} is symmetric in term of $\k\in\mathcal{B}$, in the following sense: for all $\k\in\mathcal{B}$, it holds that $-\k\in\mathcal{B}$, since we have taken $\mathcal{B} = [-\pi,\pi]$. Thus, replacing $\k$ by $-\k$ in \eqref{eq:paritysymmetry}, we get
        \begin{align}\label{eq:paritysymmetry:converse}
            \mathcal{P}\mathcal{L}(\k,\w) = \mathcal{L}(-\k,\w) \mathcal{P}.
        \end{align}
        We also know that
        \begin{align*}
            \mathcal{L}(\k,\w_n) u_n^{[\k]}(x,\w_n) = \l_n(\w_n) u_n^{[\k]}(x,\w_n).
        \end{align*}
        Hence,
        \begin{align*}
            \mathcal{P}\mathcal{L}(\k,\w_n) u_n^{[\k]}(x,\w_n) = \mathcal{P} \l_n(\w_n) u_n^{[\k]}(x,\w_n) = \l_n(\w_n) \mathcal{P} u_n^{[\k]}(x,\w_n).
        \end{align*}
        So, from \eqref{eq:paritysymmetry:converse}, we get
        \begin{align*}
            \mathcal{L}(-\k,\w_n) \mathcal{P} u_n^{[\k]}(x,\w_n) = \l_n(\w_n) \mathcal{P} u_n^{[\k]}(x,\w_n),
        \end{align*}
        which gives the desired result.
        
        \item Let $\l_n(-\k,\w_n)$ be an eigenvalue of $\mathcal{L}(-\k,\w_n)$, with $u_n^{[-\k]}(x,\w_n)$ being the associated eigenvector, for $\w_n\in\mathbb{R}$ and $\k\in\mathcal{B}$. This implies that
        \begin{align*}
            \mathcal{L}(-\k,\w_n) u_n^{[-\k]}(x,\w_n) = \l_n(-\k,\w_n) u_n^{[-\k]}(x,\w_n).
        \end{align*}
        Although, we know that $\w_n(-\k)=\w_n(\k)$ and we have that $\l_n(\k,\w_n) = \w_n^2(\k)$. Hence, we get
        \begin{align*}
            \l_n(\k,\w_n) = \l_n(-\k,\w_n),
        \end{align*}
        which gives the desired result.
        
    \end{itemize}
    This concludes the proof.
\end{proof}

\begin{lemma}\label{lemma:mu}
    Let $\l_n(\k,\w_n)$ be a non-degenerate eigenvalue of $\mathcal{L}(\k,\w_n)$ with $u_n^{[\k]}(x,\w_n)$ being the associated eigenvector, for $\k$ at the band edges, i.e. $\k\in\{0,\pm\pi\}$. Then,
    \begin{align*}
        u_n^{[\k]}(x,\w_n) = \mathcal{P} u_n^{[\k]}(x,\w_n) \ \ \ \text{ or } \ \ \  u_n^{[\k]}(x,\w_n) = -\mathcal{P} u_n^{[\k]}(x,\w_n), \ \ \ \text{ for } \k\in\{0,\pm\pi\}.
    \end{align*}
\end{lemma}

\begin{proof}
    Since we take $\l_n(\k,\w_n)$ to be non-degenerate, there exists $\m\in\mathbb{R}$ such that
    \begin{align}\label{eq:eigproportional}
        u_n^{[-\k]}(x,\w_n) = \m \mathcal{P} u_n^{[\k]}(x,\w_n).
    \end{align}
    Let us take $\k$ to be at the band edges $\k\in\{0,\pm\pi\}$. We get for $\k=0$
    \begin{align*}
        u_n^{[0]}(x,\w_n) = \m \mathcal{P} u_n^{[0]}(x,\w_n).
    \end{align*}
    We know that $\mathcal{P}\mathcal{P}=Id$. Thus, applying $\mathcal{P}$ on both sides, we get
    \begin{align*}
        \mathcal{P} u_n^{[0]}(x,\w_n) = \m u_n^{[0]}(x,\w_n),
    \end{align*}
    which gives
    \begin{align*}
        u_n^{[0]}(x,\w_n) = \m^2 u_n^{[0]}(x,\w_n) \Rightarrow \m^2 = 1.
    \end{align*}
    Thus, either $\m=1$, in which case $\mathcal{P} u_n^{[0]}(x,\w_n) = u_n^{[0]}(x,\w_n)$, or $\m=-1$, in which case $\mathcal{P} u_n^{[0]}(x,\w_n) = -u_n^{[0]}(x,\w_n)$.\\
    Let us now take $\k=-\pi$. We have
    \begin{align*}
        u_n^{[\pi]}(x,\w_n) = \m \mathcal{P} u_n^{[-\pi]}(x,\w_n).
    \end{align*}
    Taking $\k = \pi$, we have
    \begin{align*}
        u_n^{[-\pi]}(x,\w_n) = \m \mathcal{P} u_n^{[\pi]}(x,\w_n).
    \end{align*}
    Combining these two expressions, we obtain 
    \begin{align*}
        u_n^{[\pi]}(x,\w_n) = \m^2 \mathcal{P} u_n^{[\pi]}(x,\w_n) \Rightarrow \m^2=1
    \end{align*}
    and
    \begin{align*}
        u_n^{[-\pi]}(x,\w_n) = \m^2 \mathcal{P} u_n^{[-\pi]}(x,\w_n) \Rightarrow \m^2=1.
    \end{align*}
    Thus, either $\m=1$, in which case $\mathcal{P} u_n^{[\pi]}(x,\w_n) = u_n^{[\pi]}(x,\w_n)$ and $\mathcal{P} u_n^{[-\pi]}(x,\w_n) = u_n^{[-\pi]}(x,\w_n)$, or $\m=-1$, in which case $\mathcal{P} u_n^{[\pi]}(x,\w_n) = -u_n^{[\pi]}(x,\w_n)$ and $\mathcal{P} u_n^{[-\pi]}(x,\w_n) = -u_n^{[-\pi]}(x,\w_n)$. This concludes the proof.
\end{proof}

We will call a function symmetric if $f = \mathcal{P}f$ and anti-symmetric if $f = -\mathcal{P}f$. Then, it is well known that due to the symmetry of the unit cell, combined with the periodicity or anti-periodicity that occurs when $\k\in\{0,\pi\}$, the modes at the edges of the bands must be either symmetric or anti-symmetric (see also \cite{coutant2023surface}).

\begin{lemma}\label{lemma:modesymmetry}
    Let $\k\in\{0,\pi\}$ and let $u_n'$ denote the spatial derivative of $u_n$. Then, either
    \begin{center}
        $u_n$ is symmetric and $u'_n$ is anti-symmetric,
    \end{center}
    or
    \begin{center}
        $u_n$ is anti-symmetric and $u'_n$ is symmetric.
    \end{center}
\end{lemma}

\begin{proof}
    Let us consider $\k$ at the band edges $0$ and $\pi$. Then, we know that either $u_n^{[\k]}(x,\w_n) = \mathcal{P}u_n^{[\k]}(x,\w_n)$ or $u_n^{[\k]}(x,\w_n) = -\mathcal{P}u_n^{[\k]}(x,\w_n)$. In the first case, we see that $u_n$ is symmetric. In addition,
    \begin{align*}
        &u_n^{[\k]}(x,\w_n) = \mathcal{P}u_n^{[\k]}(x,\w_n) = u_n^{[\k]}(1-x,\w_n) \Rightarrow \\
        &\Big(u_n^{[\k]}\Big)'(x,\w_n) = - \Big(u_n^{[\k]}\Big)'(1-x,\w_n) = - \mathcal{P}\Big(u_n^{[\k]}\Big)'(x,\w_n),
    \end{align*}
    which shows that $u'_n$ is anti-symmetric. Similarly, in the second case, we have that $u'_n$ is anti-symmetric. Then,
    \begin{align*}
        u_n^{[\k]}(x,\w_n) = -u_n^{[\k]}(1-x,\w_n) \Rightarrow \Big(u_n^{[\k]}\Big)'(x,\w_n) = \Big(u_n^{[\k]}\Big)'(x,\w_n) = \mathcal{P}\Big(u_n^{[\k]}\Big)'(x,\w_n),
    \end{align*}
    which shows that $u'_n$ is symmetric.
\end{proof}

Since the modes are labeled by $n$, from now on, in order to ease the notation, we will suppress the dependence on the frequency $\w_n$, since it appears indirectly at the subscripts, i.e. $u_n^{[\k]}(x,\w_n) \equiv u_n^{[\k]}(x)$. A consequence of the symmetries characterised in Lemma~\ref{lemma:modesymmetry} is that the modes at the edges of the spectral bands may have critical points at $x=0$, which can be characterised based on the modes' symmetries. These symmetry arguments are similar to those in \cite{coutant2023surface}. A specific proof for our setting is given in Appendix~\ref{app:symmetries}.

\begin{lemma}\label{lemma:u,u'}
    Let $\k\in\{0,\pi\}$. Then, either
        $$u_n^{[\pi]}(0) = 0 \quad\text{and}\quad\Big(u_n^{[0]}\Big)'(0) = 0,$$
    or
    \begin{center}
        $$u_n^{[0]}(0) = 0\quad\text{and}\quad\Big(u_n^{[\pi]}\Big)'(0) = 0.$$
    \end{center}
\end{lemma}

\subsection{Bulk index}

Let us now define the notion of the bulk index. This is a topological property of the material which depends on the symmetry of the mode at the edges of a band.

\begin{definition}[Bulk index]
    Let $\mathfrak{S}:=[a,b]$ denote the $n$-th band gap. Then we define the bulk topological index $\mathcal{J}_n$ of the $n$-th band gap as follows:
    \begin{align}\label{def:bulk}
        \mathcal{J}_n :=
        \begin{cases}
            +1, \ \ \ \text{ if } u_n \text{ is symmetric at $a$},\\
            -1, \ \ \ \text{ if } u_n \text{ is anti-symmetric at $a$}.
        \end{cases}
    \end{align}
\end{definition}

\subsection{Frequency existence}

Using the symmetry properties that we have proved, we will show that there exists $\w\in\mathbb{R}$, such that \eqref{eq:impedancecondition} is satisfied.

\subsubsection{Impedance evaluation}

Let us now see what the impedance functions look like at the band edges $\k\in\{0,\pi\}$. We introduce the following notation:
\begin{align*}
    Z^{\pm}_{\k} := Z^{\pm} \text{ at } \k,
\end{align*}
for $\k\in\mathfrak{B}$. One of the main insights from the work of \cite{coutant2023surface} is the dependence of the surface impedance on the symmetries of the eigenmodes (and, by definition, on the bulk index). This can be replicated in our setting with the following lemma.

\begin{lemma}\label{lemma:Z:eval}
    It holds that either
    \begin{center}
        $Z^{\pm}_{0} = \pm\infty$ and $Z^{\pm}_{\pi} = 0$,
    \end{center}
    or 
    \begin{center}
        $Z^{\pm}_{0} = 0$ and $Z^{\pm}_{\pi} = \pm\infty$.
    \end{center}
\end{lemma}

\begin{proof}
    As mentioned above, from the uniqueness of solution to \eqref{eq:Helmholtz}, we know that, for $\k\in\{0,\pi\}$, we cannot have $u_n^{[\k]}(0) = \Big( u_n^{[\k]} \Big)'(0)=0$. Also, from Lemma \ref{lemma:u,u'}, we have two cases for the values of $u_n$ and $u_n'$ at the band edges. We shall treat them separately.
    \begin{itemize}
    
        \item In the first case, we know that
        \begin{align*}
            u_n^{[\pi]}(0) = 0 \ \ \ \text{and} \ \ \ \Big(u_n^{[0]}\Big)'(0) = 0.
        \end{align*}
        This implies that
        \begin{align*}
            \Big(u_n^{[\pi]}\Big)'(0) \ne 0 \ \ \ \text{and} \ \ \ u_n^{[0]}(0) \ne 0.
        \end{align*}
        Hence, combining these results at $\k=0$, we get
        \begin{align*}
            \lim_{x\uparrow0} Z^-_0(\w) = \lim_{x\uparrow0} -\frac{u_n^{[0]}(x,\w)}{\frac{1}{\ve_1(\w)}\Big(u_n^{[0]}\Big)'(x,\w)} = \pm\infty \ \ \text{ and } \ \ \lim_{x\downarrow0} Z^+_0(\w) = \lim_{x\downarrow0} -\frac{u_n^{[0]}(x,\w)}{\frac{1}{\ve_1(\w)}\Big(u_n^{[0]}\Big)'(x,\w)} = \pm\infty.
        \end{align*}
        At $\k=\pi$, we see that
        \begin{align*}
            \lim_{x\uparrow0} Z^-_{\pi}(\w) = \lim_{x\uparrow0} -\frac{u_n^{[\pi]}(x,\w)}{\frac{1}{\ve_1(\w)}\Big(u_n^{[\pi]}\Big)'(x,\w)} = 0 \ \ \text{ and } \ \ \lim_{x\uparrow0} Z^+_{\pi}(\w) = \lim_{x\uparrow0} -\frac{u_n^{[\pi]}(x,\w)}{\frac{1}{\ve_1(\w)}\Big(u_n^{[\pi]}\Big)'(x,\w)} = 0.
        \end{align*}

        \item For the second case, we know that
        \begin{align*}
            u_n^{[0]}(0) = 0 \ \ \ \text{and} \ \ \  \Big(u_n^{[\pi]}\Big)'(0) = 0,
        \end{align*}
        which implies that
        \begin{align*}
            \Big(u_n^{[0]}\Big)'(0) \ne 0 \ \ \ \text{and} \ \ \  u_n^{[\pi]}(0) \ne 0.
        \end{align*}
        In the exact same reasoning as in the previous case, we get at $\k=0$,
        \begin{align*}
            \lim_{x\uparrow0} Z^-_0(\w) = \lim_{x\uparrow0} -\frac{u_n^{[0]}(x,\w)}{\frac{1}{\ve_0}\Big(u_n^{[0]}\Big)'(x,\w)} = 0 \ \ \text{ and } \ \ \lim_{x\downarrow0} Z^+_0(\w) = \lim_{x\downarrow0} -\frac{u_n^{[0]}(x,\w)}{\frac{1}{\ve_0}\Big(u_n^{[0]}\Big)'(x,\w)} = 0
        \end{align*}
        and at $\k=\pi$,
        \begin{align*}
            \lim_{x\uparrow0} Z^-_{\pi}(\w) = \lim_{x\uparrow0} -\frac{u_n^{[\pi]}(x,\w)}{\frac{1}{\ve_0}\Big(u_n^{[\pi]}\Big)'(x,\w)} = \pm\infty \ \ \text{ and } \ \ \lim_{x\uparrow0} Z^+_{\pi}(\w) = \lim_{x\uparrow0} -\frac{u_n^{[\pi]}(x,\w)}{\frac{1}{\ve_0}\Big(u_n^{[\pi]}\Big)'(x,\w)} = \pm\infty.
        \end{align*}
        
    \end{itemize}
    This concludes the proof.
\end{proof}

\subsubsection{Impedance property}

In order to ease the notation, let us define:
\begin{align}\label{def:E,phi}
    E(x,\w) := \frac{1}{\ve(x,\w)} \ \ \ \text{ and } \ \ \ \phi := \frac{\upd u}{\upd \w}.
\end{align}

The following result is due to the work of \cite{coutant2023surface}, which says that the surface impedance functions are strictly decreasing functions of $\omega$, within each band gap. For dispersive systems, this rests on the aforementioned assumption that $\epsilon$ is a non-decreasing function of $\omega$. A proof of this result, specific to our setting, is given in Appendix~\ref{app:monotonicity}.

\begin{theorem}\label{thm:Z:decrease}
    Let us assume that $\w\in\mathfrak{A}$, where $\mathfrak{A}\subset\mathbb{R}$ is a band gap. Then, the surface impedance decreases with respect to the frequency, i.e.
    \begin{align*}
        \frac{\upd Z^+}{\upd \w}<0, \ \ \frac{\upd Z^-}{\upd \w}<0, \ \ \ \text{ for } \w\in\mathfrak{A}.
    \end{align*}
\end{theorem}

\subsubsection{Interface mode existence}

Let $\mathcal{L}_A$ and $\mathcal{L}_B$ denote the differential operator $\mathcal{L}$ for the material $A$ and the material $B$, respectively. We define $\mathfrak{A}_n^{A}$ and $\mathfrak{A}_n^{B}$ to be the $n$-th band gaps of $\mathcal{L}_A$ and $\mathcal{L}_B$, respectively. Let them be given by
\begin{align*}
    \mathfrak{A}_n^{A} := [\w_{n,A}^+,\w_{n+1,A}^-] \ \ \ \text{ and } \ \ \ \mathfrak{A}_n^{B} := [\w_{n,B}^+,\w_{n+1,B}^-].
\end{align*}
Let us denote by $\mathfrak{A}_n$ the intersection of $\mathfrak{A}_n^{A}$ and $\mathfrak{A}_n^{B}$. We assume that it is non-empty, i.e.
\begin{align*}
    \mathfrak{A}_n = \mathfrak{A}_n^{A} \cap \mathfrak{A}_n^{B} \ne \emptyset.
\end{align*}
We denote it by
\begin{align*}
    \mathfrak{A}_n := [\w_{n}^+,\w_{n+1}^-].
\end{align*}
Finally, let us denote by $\mathcal{J}_{n,A}$ and $\mathcal{J}_{n,B}$ the bulk topological indices associated to the material $A$ and the material $B$, respectively, in $\mathfrak{A}_n$.

\begin{theorem}\label{thm:bulksum}
    If
    \begin{align*}
        \mathcal{J}_{n,A} + \mathcal{J}_{n,B} \ne 0,
    \end{align*}
    then no interface mode exists. If 
    \begin{align*}
        \mathcal{J}_{n,A} + \mathcal{J}_{n,B} = 0,
    \end{align*}
    then there exists a unique frequency $\w_{m}\in\mathfrak{A}_n$, for which an interface mode exists.
\end{theorem}
\begin{proof}
    We will divide the proof in two steps, one for each case.
    \begin{itemize}
        
        \item Let us consider:
        \begin{align*}
            \mathcal{J}_{n,A} + \mathcal{J}_{n,B} \ne 0.
        \end{align*}
        We will treat the case $\k=0$. The same argument holds for $\k=\pi$. From \eqref{def:bulk}, this implies that either $u^+_{n,A}, u^+_{n,B}$ are both symmetric or both anti-symmetric. Let us treat the case of both being symmetric. The anti-symmetric case follows the same reasoning. From Theorem \ref{thm:Z:decrease}, we have that $Z^-$ is decreasing in $\mathfrak{A}_n$. Then, from Lemma \ref{lemma:Z:eval}, we have
        \begin{align*}
            Z^-_0(\w_{n,A}^+) = +\infty \searrow Z^-_0(\w_{n+1,A}^-) = 0 \ \text{ in } \mathfrak{A}_{n}^A.
        \end{align*}
        Similarly, we get
        \begin{align*}
            Z^+_0(\w_{n,B}^+) = +\infty \searrow Z^+_0(\w_{n+1,B}^-) = 0 \ \text{ in } \mathfrak{A}_{n}^B.
        \end{align*}
        We see directly that Lemma \ref{lemma:impedance} does not hold, and hence, no interface mode exists.

        \item Let us consider: 
        \begin{align*}
            \mathcal{J}_{n,A} + \mathcal{J}_{n,B} = 0.
        \end{align*}
        We will treat the case $\k=0$. The same argument holds for $\k=\pi$. From \eqref{def:bulk}, this implies that either $u^+_{n,A}$ is symmetric and $u_{n,B}^+$ is anti-symmetric, or $u^+_{n,A}$ is anti-symmetric and $u_{n,B}^+$ is symmetric. Let us focus on the first possibility. For the second one, the same argument will hold. So, we take
        \begin{center}
            $u^+_{n,A}$ symmetric and $u_{n,B}^+$ anti-symmetric at $\k=0$.
        \end{center}
        Then, from Lemma \ref{lemma:Z:eval} and Theorem \ref{thm:Z:decrease}, we see that 
        \begin{align*}
            Z^-_0(\w_{n,A}^+) = +\infty \searrow Z^-_0(\w_{n+1,A}^-) = 0 \  \text{ in } \mathfrak{A}_{n}^A.
        \end{align*}
        Using the fact that $u_{n,B}^+$ is anti-symmetric, we get, from Lemma \ref{lemma:Z:eval} and Theorem \ref{thm:Z:decrease}, that
        \begin{align*}
            Z^+_0(\w_{n,B}^+) = 0 \searrow Z^+_0(\w_{n+1,B}^-) = -\infty \ \text{ in } \mathfrak{A}_{n}^B.
        \end{align*}
        Let us now study how $Z_0:=Z^-_0+Z^+_0$ behaves in $\mathfrak{A}_n$. Since $\mathfrak{A}_n\ne\emptyset$, we observe that
        \begin{align*}
            \min{\mathfrak{A}_n} \in \Big\{ \w_{n,A}^+, \w_{n,B}^+\Big\}.
        \end{align*}
        If $\min{\mathfrak{A}_n} = \w_{n,A}^+$, then we see that 
        \begin{align*}
            Z_0\Big(\w_{n,A}^+\Big) = +\infty,
        \end{align*}
        since $Z^-_0\Big(\w_{n,A}^+\Big)=+\infty$ and $Z^+_0\Big(\w_{n,A}^+\Big)\in\mathbb{R}_{<0}$. Also, if $\min{\mathfrak{A}_n} = \w_{n,B}^+$, then we see that 
        \begin{align*}
            Z_0\Big(\w_{n,B}^+\Big) >0,
        \end{align*}
        since $Z^-_0\Big(\w_{n,B}^+\Big)\in\mathbb{R}_{>0}$ and $Z^+_0\Big(\w_{n,B}^+\Big)=0$. Thus, in any case, we get that
        \begin{align*}
            Z_0\Big(\min{\mathfrak{A}_n}\Big) >0.
        \end{align*}
        Similarly, we have that 
        \begin{align*}
            \max{\mathfrak{A}_n} \in \Big\{ \w_{n+1,A}^-, \w_{n+1,B}^-\Big\}.
        \end{align*}
        Then, if $\max{\mathfrak{A}_n} = \w_{n+1,A}^-$, we see that
        \begin{align*}
            Z_0\Big(\w_{n+1,A}^-\Big) < 0,
        \end{align*}
        since $Z^-_0\Big(\w_{n+1,A}^-\Big)=0$ and $Z^+_0\Big(\w_{n+1,A}^-\Big)\in\mathbb{R}_{<0}$. Also, if $\max{\mathfrak{A}_n} = \w_{n+1,B}^-$, we have that
        \begin{align*}
            Z_0\Big(\w_{n+1,A}^-\Big) = - \infty,
        \end{align*}
        since $Z^-_0\Big(\w_{n+1,A}^-\Big)\in\mathbb{R}_{>0}$ and $Z^+_0\Big(\w_{n+1,A}^-\Big)=-\infty$. Hence, in any case, we observe that
        \begin{align*}
             Z_0\Big(\max{\mathfrak{A}_n}\Big) < 0.
        \end{align*}
        To sum up these results, we have that
        \begin{align*}
            Z_0\Big(\min{\mathfrak{A}_n}\Big) >0 \ \ \ \text{ and } \ \ \ Z_0\Big(\max{\mathfrak{A}_n}\Big) < 0.
        \end{align*}
        Then, from Theorem \ref{thm:Z:decrease}, we deduce that
        \begin{align}
            \frac{\upd Z_0}{\upd \w}<0 \ \text{ in } \mathfrak{A}_n.
        \end{align}
        So, since $Z_0$ is continuous and strictly decreasing in $\mathfrak{A}_n$, there exists a unique point $\w_m\in\mathfrak{A}_n$ at which $Z_0(\w_m)=0$. Therefore, a unique interface mode exists. This concludes the proof.
        
    \end{itemize}
\end{proof}

\begin{remark}
    Here, let us recall that we have considered piecewise constant permittivities with respect to $x$ but still dispersive with respect to $\w$. Although, as mentioned in Section \ref{section:MS}, the analysis still applies on permittivity functions which respect the mirror symmetry in each periodic cell and the results remain the same. 
\end{remark}

\section{Zak phase}\label{section:ZP}

Let us now study the Zak phase, an invariant associated to one-dimensional crystals. This describes the topological properties of our system.

As mentioned at the beginning, we work in the space $L^2_{\k}$. We equip $L^2_{\k}$, with the following inner product:
\begin{align}\label{def:innerproduct}
    \langle f,g \rangle_ := \int_{0}^1 \m_0 f(x) \overline{g}(x) \upd x, \ \ \ f,g\in L^2_{\k}.
\end{align}

\begin{definition}[Berry connection]
    We define the Berry connection $A_n(\k)$ of the $n$-th band to be
    \begin{align}\label{def:Berry}
        A_n(\k) := -i \left\langle \frac{\partial u^{[\k]}_n}{\partial\k}, u_n^{[\k]} \right\rangle.
    \end{align}
\end{definition}

\begin{definition}[Zak phase]
    We define the Zak phase $\Theta_n$ of the $n$-th gap to be the integral of the Berry connection $A_n(\k)$ across the first Brillouin zone, i.e.,
    \begin{align}\label{def:Zak}
        \Theta_n = \int_{-\pi}^{\pi} A_{n}(\k) \upd\k.
    \end{align}
\end{definition}



To see that the Zak phase is well defined, we want to check that it is invariant $\mathrm{mod}(2\pi)$ with respect to phase changes of the eigenmode $u_n^{[\k]}$. This is because a normalised eigenmode (with $\|u_n^{[\k]}\|=1$) can always be re-defined by changing the phase by an integer multiple of $\k$,
    \begin{align*}
        \tilde{u}^{[\k]}_n = e^{i\z\k}u^{[\k]}_n, \ \z\in\mathbb{N}.
    \end{align*}
    Under this transformation, we can calculate that
    \begin{align*}
        \tilde{A}_n(\k) &= -i \left\langle \frac{\partial \tilde{u}^{[\k]}_n}{\partial x}, \tilde{u}^{[\k]}_n \right\rangle = 
        -i \left\langle \frac{\partial}{\partial x}\Big(e^{i\z\k }u^{[\k]}_n\Big), e^{i\z\k } u^{[\k]}_n \right\rangle \\
        &=-i \int_0^1 \m_0 \left( e^{i\z\k }\frac{\partial u^{[\k]}_n}{\partial x}(x) + i \z  e^{i\z\k }u^{[\k]}_n(x) \right) \overline{ e^{i\z\k }u^{[\k]}_n(x) } \upd x \\
        &= A_n(\k) + \int_{0}^1 \m_0 \z  u^{[\k]}_n(x) \overline{u^{[\k]}_n(x)} \upd x\\
        &= A_n(\k) + \z.
    \end{align*}
    Thus,
    \begin{align*}
        \int_{-\pi}^{\pi} \tilde{A}_n(\k) \upd\k =  \int_{-\pi}^{\pi} A_n(\k) \upd\k + 2\z\pi,
    \end{align*}
    which gives
    \begin{align}
        \tilde{\Theta}_n(\k) = \Theta_n(\k) + 2\z\pi,
    \end{align}
so $\tilde{\Theta}_n$ and ${\Theta}_n$ are equal $\mathrm{mod}(2\pi)$.


\begin{theorem}
    Let us assume that we have mirror symmetry and that the $n$-th eigenvalue of the operator $\mathcal{L}$ is non-degenerate. Then, it holds that $\Theta_n\mathrm{mod}(2\pi) \in \{0,\pi\}$.
\end{theorem}
\begin{proof}
    We recall that, for the assumptions for the theorem, $u_n^{[-\k]}$ and $\mathcal{P}u_n^{[\k]}$ are proportional to one another, i.e., there exists $\m\in\mathbb{R}$, depending on $\k$, such that
    \begin{align*}
        u_n^{[-\k]} = \m(\k) \mathcal{P}u_n^{[\k]}.
    \end{align*}
    In addition, we have taken the modes to satisfy $\|u_n\|=1$. Hence, we get $\|\m(\k)\|=1$ and so, we can write
    \begin{align*}
        u_n^{[-\k]} = e^{i\d(\k)} \mathcal{P}u_n^{[\k]},
    \end{align*}
    where $\d(\k)$ is a locally smooth function of $\k$. Then, we see that
    \begin{align*}
            \frac{\upd}{\upd\k} u_n^{[-\k]} = \frac{\upd}{\upd\k} \Big( e^{i\d(\k)} \mathcal{P}u_n^{[\k]} \Big) \Rightarrow \frac{\upd}{\upd\k} u_n^{[-\k]} = ie^{i\d(\k)}\frac{\upd \d}{\upd\k}\mathcal{P}u_n^{[\k]} + \frac{\upd\mathcal{P}u_n^{[\k]}}{\upd\k}e^{i\d(\k)}.
    \end{align*}
    Taking the inner product with $u_n^{[-\k]}$ on both sides, we see that
    \begin{align*}
        \int_{0}^1 \m_0e^{i\d(\k)}\frac{\upd\mathcal{P}u_n^{[\k]}}{\upd\k}(x)\overline{u_n^{[-\k]}}(x) \upd x &= \int_{0}^1 \m_0 e^{i\d(\k)} \frac{\upd\mathcal{P}u_n^{[\k]}}{\upd\k}(x) \overline{e^{i\d(\k)}\mathcal{P}u_n^{[\k]}}(x)  \upd x\\
        &= \int_{0}^1 \m_0  \frac{\upd u_n^{[\k]}}{\upd\k}(1-x) \overline{ u_n^{[\k]}} (1-x)  \upd x\\
        &= - \int_{1}^0 \m_0 \frac{\upd u_n^{[\k]}}{\upd\k}(x) \overline{ u_n^{[\k]}}(x)  \upd x = -i A_n(\k)
    \end{align*}
    and 
    \begin{align*}
        \int_0^1 \m_0 e^{i\d(\k)} \frac{\upd \d}{\upd\k} \mathcal{P}u_n^{[\k]} (x)\overline{u_n^{[-\k]}}(x) \upd x = \frac{\upd \d}{\upd\k} \int_0^1 \m_0 u_n^{[-\k]} (x)\overline{u_n^{[-\k]}}(x) \upd x = i\frac{\upd \d}{\upd\k}. 
    \end{align*}
    Thus, we get 
    \begin{align*}
        iA_n(-\k) = -iA_n(\k) + i\frac{\upd \d}{\upd\k}. 
    \end{align*}
    Multiplying by $i$ and integrating over $[0,\pi]$, we get
    \begin{align*}
        -\int^{\pi}_0 A_n(-\k) \upd\k = \int_0^{\pi }A_n(\k) \upd\k - \d(\pi) + \d(0),
    \end{align*}
    which is
    \begin{align*}
        -\int_{-\pi}^0 A_n(\k) \upd\k = \int_0^{\pi }A_n(\k) \upd\k - \d(\pi) + \d(0).
    \end{align*}
    Thus, we get
    \begin{align*}
        \Theta_n = \d(\pi) - \d(0).
    \end{align*}
    From Lemma \ref{lemma:mu}, we have that $\m(\k)\in\{\pm,1\}$, for $\k=0,\pi$. This implies that $\d(\k)\mathrm{mod}(2\pi)\in\{0,\pi\}$, for $\k=0,\pi$. If $\d(\pi)\mathrm{mod}(2\pi) = \d(0)\mathrm{mod}(2\pi)$, then $\Theta_n=0$. If $\d(\pi)\mathrm{mod}(2\pi) \ne \d(0)\mathrm{mod}(2\pi)$, then 
    the phase difference is always $\pi$. Thus, we get
    \begin{align*}
        \Theta_n\mathrm{mod}(2\pi) \in \{0,\pi\}.
    \end{align*}
    This concludes the proof.
\end{proof}

\section{Asymptotic behaviour}\label{section:AB}

We wish to study the asymptotic behaviour of the modes as $x\to\pm\infty$. For this, we will make use of the transfer matrix associated to this problem. 

In the sections to follow, we will consider permittivity functions $\ve(x,\w)$ which are piecewise constant with respect to $x$ in each particle. This method still applies for permittivities which are mirror symmetric with respect to $x$, but by suppressing this dependence, we obtain explicit expressions for the transfer matrices, and so, we have a more qualitative result.

\subsection{Transfer matrix method}\label{subsection:Transfer_matrix_method}

The transfer matrix method is a way of describing the mode $u$ and its spatial derivative $u'$ at each point on the structure with respect to the initial data vector $\Big(u(0),u'(0)\Big)^{\top}$.

We define the segment matrices $\mathscr{T}(x,\w)$ by
\begin{align}\label{def:segmentmatrix}
    \mathscr{T}(l_x,\w) := 
    \begin{pmatrix}
        \cos(\sqrt{\m_0\ve(x,\w)}\w l_x) & \frac{\sin(\sqrt{\m_0\ve(x,\w)}\w l_x)}{\sqrt{\m_0\ve(x,\w)}\w} \\
        -\sqrt{\m_0\ve(x,\w)}\w \sin(\sqrt{\m_0\ve(x,\w)}\w l_x) & \cos(\sqrt{\m_0\ve(x,\w)}\w l_x)
    \end{pmatrix},
\end{align}
where $l_x$ denotes the length of the part of the segment on which $x$ lies. 

Let us consider an interval $I$ which is constituted by $N$ segments of length $l_i$ each, $i=1,\dots,N$. Then, we define the transfer matrix $T_I(\w)$ over the interval $I$ to be
\begin{align}\label{def:transfermatrix}
    T_I(\w) := \prod_{i=1}^N \mathscr{T}(l_i,\w).
\end{align}
If we define the vector $\tilde{u}(x) := \Big( u(x) , u'(x) \Big)^\top$, then the transfer matrix method describes the mode $u$ at each point $x$ based on the initial data at a point $x_0$, i.e.
\begin{align}
    \tilde{u}(x) = T_{[x_0,x]}(\w)  \tilde{u}(x_0).
\end{align}
It is clear that the segment matrices $\mathscr{T}$ satisfy $\det(\mathscr{T}(x,\w))=1$ for any $x$ and $\omega$. Hence, also, the transfer matrix $T$ satisfies $\det(T(x,\w))\equiv 1$.

The mirror symmetry induced to the system gives
\begin{align*}
    \ve(x_n + h,\w) = \ve(x_{n+1}-h,\w), \ \ h\in[0,1),\  n\in\mathbb{N}\setminus\{0\}.
\end{align*}
Let $T^{[j]}_p$ denote the transfer matrix over one periodic cell of the material $j$, with $j\in\{A,B\}$. Then, it holds, as in \cite{craster2023asymptotic, schoenberg1983properties}, that
\begin{align}\label{eq:transfer}
    \begin{pmatrix}
        u(x_{n+1}) \\ u'(x_{n+1})
    \end{pmatrix}
    = \mathbbm{1}_{\{n\geq0\}} T_p^{[B]}
    \begin{pmatrix}
        u(x_{n}) \\ u'(x_{n})
    \end{pmatrix}
    + \mathbbm{1}_{\{n<0\}} T_p^{[A]}
    \begin{pmatrix}
        u(x_{n}) \\ u'(x_{n})
    \end{pmatrix},
\end{align}
for $n\in\mathbb{Z}$. Also, the symmetry of the system gives
\begin{align}\label{eq:inversetransfer}
    \begin{pmatrix}
        u(x_{n-1}) \\ u'(x_{n-1})
    \end{pmatrix}
    = \mathbbm{1}_{\{n\geq0\}} S T_p^{[B]} S
    \begin{pmatrix}
        u(x_{n}) \\ u'(x_{n})
    \end{pmatrix}
    + \mathbbm{1}_{\{n<0\}} S T_p^{[A]} S
    \begin{pmatrix}
        u(x_{n}) \\ u'(x_{n})
    \end{pmatrix},
\end{align}
for $n\in\mathbb{Z}$, where the matrix $S$ is given by
\begin{align*}
    S := \begin{pmatrix}
       1 & 0 \\
       0 & -1
    \end{pmatrix}.
\end{align*}
Here, let us note that with a direct calculation, we can see that $S=S^{-1}$. \\
Applying the quasiperiodic boundary conditions, we get, for each material $i$,
\begin{align*}
    \tilde{u}(x_{n+1}) = e^{i\k} \tilde{u}(x_n).
\end{align*}
Combining this with \eqref{eq:transfer}, we get the following problem
\begin{align}
    \Big( T^{[j]}_p(\w) - e^{i\k} I \Big) \tilde{u}(x_n) = 0.
\end{align}

\subsubsection{Spectral properties}

Let us now state certain spectral properties of the transfer matrices.

\begin{lemma}
    Let $\w\in\mathbb{R}\setminus\{\w_p^{(\pm)}\}$. Then, $T^{[j]}_p$ has real eigenvalues denoted by $\l^{[j]}_1$ and $\l^{[j]}_2$, satisfying $|\l^{[j]}_1|<1$ and $|\l^{[j]}_2|>1$, for $j\in\{A,B\}$.
\end{lemma}
\begin{proof}
    Let $i\in\{A,B\}$. Since $\det(T^{[j]}_p(\w))=1$, we get either
    \begin{align*}
        \l^{[j]}_1,\l^{[j]}_2\in\mathbb{R}, \ \ \text{with } \ |\l^{[j]}_1|<1 \ \text{ and } \ |\l^{[j]}_2|>1,
    \end{align*}
    or
    \begin{align*}
        \l^{[j]}_1,\l^{[j]}_2\in\mathbb{C}, \ \ \text{with } \ |\l^{[j]}_1|=|\l^{[j]}_2|=1 \ \text{ and } \ \l^{[j]}_1 = \overline{\l^{[j]}_2}.
    \end{align*}
    If the second case holds, then there exists $\k$ such that $\l^{[j]}_1=e^{i\k}$, which implies that $\det\Big(T^{[j]}_p(\w) - e^{i\k} I\Big)=0$ which gives a contradiction. Hence, the first case holds, and so, we obtain the desired result.
\end{proof}
We use this result to obtain information about the asymptotic behaviour of the modes $u(x_n)$ as $n\to\pm\infty$.

\begin{theorem}\label{thm:AB:spectral}
    The eigenfrequency $\w$ of a localised eigenmode of the Helmholtz problem \eqref{eq:Helmholtz} must satisfy
    \begin{align}
        \begin{pmatrix}
            -V_{21}^{[B]}(\w) & V_{11}^{[B]}(\w)
        \end{pmatrix}
        \begin{pmatrix}
            V_{11}^{[A]}(\w) \\ -V_{21}^{[A]}(\w)
        \end{pmatrix}
        = 0,
    \end{align}
    where $(V_{11}^{[A]}(\w) , V_{21}^{[A]}(\w))^{\top}$ is the eigenvector of the transfer matrix $T_p^{[A]}$ associated to the eigenvalue $|\l_1^{[A]}|<1$ and $(V_{11}^{[B]}(\w) , V_{21}^{[B]}(\w))^{\top}$ is the eigenvector of the transfer matrix $T_p^{[B]}$ associated to the eigenvalue $|\l_1^{[B]}|<1$.
\end{theorem}

\begin{proof}
    For $n\in\mathbb{N}_{>0}$, we have
    \begin{align}\label{eq:AB:n>0}
        \begin{pmatrix}
            u(x_n)\\u'(x_n)
        \end{pmatrix}
        = (T^{[B]}_p)^n(\w) 
        \begin{pmatrix}
            u(x_0)\\u'(x_0)
        \end{pmatrix}
        = V^{[B]}
        \begin{pmatrix}
            (\l^{[B]}_1)^n & 0\\
            0 & (\l^{[B]}_2)^n
        \end{pmatrix}
        (V^{[B]})^{-1}
        \begin{pmatrix}
            u(x_0)\\u'(x_0)
        \end{pmatrix},
    \end{align}
    and for $n\in\mathbb{N}_{<0}$, from the symmetry condition \eqref{eq:inversetransfer}, we have
    \begin{align}\label{eq:AB:n<0}
        \begin{pmatrix}
            u(x_n)\\u'(x_n)
        \end{pmatrix}
        &= S T^{[A]}_p S
        \begin{pmatrix}
            u(x_{n+1})\\u'(x_{n+1})
        \end{pmatrix}
        = S (T^{[A]}_p)^{|n|} S
        \begin{pmatrix}
            u(x_{0})\\u'(x_{0})
        \end{pmatrix}\\
        &= S V^{[A]} 
        \begin{pmatrix}
            (\l^{[A]}_1)^n & 0\\
            0 & (\l^{[A]}_1)^n
        \end{pmatrix}
        (V^{[A]})^{-1} S
        \begin{pmatrix}
            u(x_{0})\\u'(x_{0})
        \end{pmatrix},
    \end{align}
    where $\l^{[A]}_1,\l^{[A]}_2$, resp. $\l^{[B]}_1,\l^{[B]}_2$, are the eigenvalues of $T^{[A]}_p$, resp. $T^{[B]}_p$, with $|\l^{[A]}_1|<1$ and $|\l^{[A]}_2|>1$, resp. $|\l^{[B]}_1|<1$ and $|\l^{[B]}_2|>1$, and $V^{[A]}$, resp. $V^{[B]}$, is the matrix of the associated eigenvectors. We are looking for localised eigenmodes satisfying
    \begin{align*}
        \lim_{n\to\pm\infty}u(x_n) = 0 \ \ \ \text{ and } \ \ \ \lim_{n\to\pm\infty}u'(x_n) = 0.
    \end{align*}
    We know that
    \begin{align*}
        \lim_{n\to\pm\infty}(\l^{[A]}_1)^{|n|} = \lim_{n\to\pm\infty}(\l^{[B]}_1)^{|n|} = 0 \ \ \ \text{ and } \ \ \  \lim_{n\to\pm\infty}(\l^{[A]}_2)^{|n|} = \lim_{n\to\pm\infty}(\l^{[B]}_2)^{|n|} \ne 0.
    \end{align*}
    Thus, we wish to find $\w,u(x_0)$ and $u'(x_0)$ such that
    \begin{align}\label{AB:conditions}
        \begin{pmatrix}
            -V^{[B]}_{21} & V^{[B]}_{11}
        \end{pmatrix}
        \begin{pmatrix}
            u(x_0) \\ u'(x_0)
        \end{pmatrix} = 0
        \ \ \ \text{ and } \ \ \ 
        \begin{pmatrix}
            -V^{[A]}_{21} & V^{[A]}_{11}
        \end{pmatrix}
        S
        \begin{pmatrix}
            u(x_0) \\ u'(x_0)
        \end{pmatrix} = 0,
    \end{align}
    since $(V^{[j]})^{-1}=\frac{1}{\det(V)}\begin{pmatrix}
        V^{[j]}_{22} & -V^{[j]}_{12}\\
        -V^{[j]}_{21} & V^{[j]}_{11}
    \end{pmatrix}$, $j\in\{A,B\}$ and so we multiply the eigenvalue $(\l^{[j]}_2)^{|n|}$, $j\in\{A,B\}$ by the second row vector. Now, we observe that the second equation in \eqref{AB:conditions} shows that $(u(x_0),u'(x_0))$ is proportional to $(V^{[B]}_{11},-V^{[B]}_{21})$. Applying this to the first equation in \eqref{AB:conditions}, we get
    \begin{align*}
        \begin{pmatrix}
            -V^{[B]}_{21} & V^{[B]}_{11}
        \end{pmatrix}
        \begin{pmatrix}
            V^{[A]}_{11} \\ -V^{[A]}_{21}
        \end{pmatrix} = 0,
    \end{align*}
    which gives the desired result.
\end{proof}

\subsubsection{Eigenmode decay}

From this, we obtain the result concerning the decay of the localised eigenmodes as $n\to\pm\infty$.

\begin{corollary}
    A localised eigenmode $u$ of \eqref{eq:Helmholtz}, posed on a medium constituted by two semi-infinite arrays of different halide perovskites, and its associated eigenfrequency $\w$ must satisfy 
    \begin{align*}
        u(x_n) = O\Big( |\l^{[A]}_1(\w)|^{|n|} \Big) \ \ \text{ and } \ \ u'(x_n) = O\Big( |\l^{[A]}_1(\w)|^{|n|} \Big) \ \ \text{ as } n\to-\infty,
    \end{align*}
    \begin{align*}
        u(x_n) = O\Big( |\l^{[B]}_1(\w)|^{|n|} \Big) \ \ \text{ and } \ \ u'(x_n) = O\Big( |\l^{[B]}_1(\w)|^{|n|} \Big) \ \ \text{ as } n\to+\infty,
    \end{align*}
    where $\l^{[A]}_1$ is the eigenvalue of $T^{[A]}_p$ satisfying $|\l^{[A]}_1(\w)|<1$ and $\l^{[B]}_1$ is the eigenvalue of $T^{[B]}_p$ satisfying $|\l^{[B]}_1(\w)|<1$.
\end{corollary}

\begin{proof}
    This is a direct result from the previous theorem. Indeed, from Theorem \ref{thm:AB:spectral}, we have that $(u(x_0),u'(x_0))^{\top}$ is proportional to $(V^{[B]}_{11},-V^{[B]}_{21})^{\top}$. Thus, we get from \eqref{eq:AB:n>0}, for $n\in\mathbb{Z}_{>0}$
    \begin{align*}
        \begin{pmatrix}
            u(x_n) \\ u'(x_n)
        \end{pmatrix}
        = (\l^{[B]}_1)^n(\w) 
        \begin{pmatrix}
            V^{[B]}_{11} \\ V^{[B]}_{21}    
        \end{pmatrix}
    \end{align*}
    and from \eqref{eq:AB:n<0}, for $n\in\mathbb{Z}_{<0}$,
    \begin{align*}
        \begin{pmatrix}
            u(x_n) \\ u'(x_n)
        \end{pmatrix}
        = (\l^{[A]}_1)^{|n|} S  
        \begin{pmatrix}
            V^{[A]}_{11} \\ V^{[A]}_{21}    
        \end{pmatrix}.
    \end{align*}
    This concludes the proof.
\end{proof}


\section{Robustness with respect to imperfections}\label{section:TP}

Now that we have established the existence of an interface mode in each band gap for dispersive materials, which decays at infinity, we wish to study the effect of system perturbations. This happens in two different ways:
\begin{itemize}
    \item Perturbations on the material parameters which have no effect on the symmetry inside each periodic cell. For this, we will introduce perturbations to the permittivity functions of the materials.
    \item Perturbations on the symmetry inside each periodic cell. To study this effect we will modify the length of two particles in each periodic cell, so that the symmetry breaks.
\end{itemize}

\subsection{Theoretical results}

\subsubsection{Permittivity perturbation} \label{subsubsection:theory:perm_pert}

First, we will study changes in the materials while the symmetry is preserved. To show this effect, we introduce the perturbation function $f:\mathbb{R}\to\mathbb{R}$, given by
\begin{align}\label{def:perturbation}
    f(x,\w) := 
    \begin{cases}
        &f_1(\w), \ \ x \in D^{[1]},\\
        &f_2(\w), \ \ x \in D^{[2]}.
    \end{cases}
\end{align}
The perturbation function $f$ is a piecewise smooth function of the frequency $\w$ and satisfies $\frac{\partial f}{\partial\w}<+\infty$. Then, we define $\widetilde{\ve}$ to be the perturbed permittivity of the system, given by
\begin{align*}
    \widetilde{\ve}(x,\w) := \ve(x,\w) + \d f(\w) = 
    \begin{cases}
        &\ve_1(\w) + \d f_1(\w), \ \ x \in D^{[1]},\\
        &\ve_2(\w) + \d f_2(\w), \ \ x \in D^{[2]},
    \end{cases}
\end{align*}
where $\d>0$ is the perturbation parameter. Then, we obtain the following perturbed problem
\begin{align}\label{eq:Helmholtz:pert} 
    \begin{cases}
        &\widetilde{\mathcal{L}}(\k,\w)u = \w^2 u,\\
        &u(x+1) = e^{i\k}u(x),
    \end{cases}
\end{align}
where
\begin{align}
    \widetilde{\mathcal{L}}u := -\frac{1}{\m_0} \frac{\upd}{\upd x} \left( \frac{1}{\widetilde{\ve}(x,\w)} \frac{\upd u}{\upd x} \right).
\end{align}
The following definition follows naturally from the perturbation of the system.

\begin{definition}[Perturbed impedance function]
    We define the surface impedances associated to the perturbed Helmholtz problem \eqref{eq:Helmholtz:pert} by
    \begin{align}
        \widetilde{Z}^-(\w)= -\frac{u(x_0^{-})}{\frac{1}{\widetilde{\ve}_1(\w)}\frac{\upd}{\upd x}u(x_0^{-})} \ \ \ \text{ and } \ \ \ \widetilde{Z}^+(\w)= \frac{u(x_0^{+})}{\frac{1}{\widetilde{\ve}_1(\w)}\frac{\upd}{\upd x}u(x_0^{+})}, \ \ \ \w\in\mathbb{R}.
    \end{align}
\end{definition}

We wish to show that, as $\d\to0$, an interface mode still exists. For the existence, Lemma \ref{lemma:impedance} holds for the perturbed impedance functions. In addition, we observe that the symmetry of the system has been preserved by the perturbation of the permittivity. Hence, it is sufficient to show that the perturbed impedance functions remain decreasing. Then, because of the symmetry of the system, the rest of the arguments will hold and so, the result will follow.  

\begin{theorem}\label{thm:pertZ:decrease}
    Let us assume that $\w\in\mathfrak{A}$, where $\mathfrak{A}\subset\mathbb{R}$ is a band gap. Then, the perturbed surface impedances, as $\d\to0$, are decreasing functions of the frequency $\w$, i.e.
    \begin{align*}
        \frac{\upd \widetilde{Z}^+}{\upd \w}<0, \ \ \frac{\upd \widetilde{Z}^-}{\upd \w}<0, \ \ \ \text{ for } \w\in\mathfrak{A}.
    \end{align*}
\end{theorem}

Thus, we have shown that arbitrarily small perturbations on the permittivity of the system do not affect the existence of an interface mode. In addition, using the transfer matrix method (explained in Section \ref{subsection:Transfer_matrix_method}), we can get an explicit expression of the solution to \eqref{eq:Helmholtz:pert} with respect to the initial data $(u(0),u'(0))$. Because of the continuous dependence of the solution on $\d$, it is direct that, as $\d\to0$, this solution converges to the one of $\eqref{eq:Helmholtz}$. This translates to the following lemma.

\begin{lemma}\label{lemma:IM_cvg:perm_pert}
    Let $u$ denote the asymptotically decaying interface mode of \eqref{eq:Helmholtz} in the band gap $\mathfrak{A}$ and let $u_{\d}$ denote an asymptotically decaying interface mode associated to the permittivity perturbed problem \eqref{eq:Helmholtz:pert} in the band gap $\mathfrak{A}_{\d}$. Then, if $\mathfrak{A} \cap \mathfrak{A}_{\d}\ne\empty$, it holds that
    \begin{align*}
        \lim_{\d\to0}u_\d=u.
    \end{align*}
\end{lemma}




\subsubsection{Symmetry perturbation} \label{subsubsection:theory:pos_pert}

Let us now introduce a structural perturbation to the system. This will result in breaking the symmetry inside each periodic cell. This phenomenon is easy to produce.

We have considered two semi-infinite materials $A$ and $B$, each one constituted by a unit cell repeated periodically and glued together at $x_0=0$. We have assumed that each periodic cell is made of two different kinds of particles, one with permittivity $\ve_1(\w)$ and one with permittivity $\ve_2(\w)$, denoted by $D^{[1]}_i$ and $D^{[2]}_j$, with $i=1,\dots,N+1$ and $j=1,\dots,N$, respectively. The particles are ordered in the following way inside each periodic cell:
\begin{align*}
    D^{[1]}_1 \ - \ D^{[2]}_1 \ - \ D^{[1]}_2 \ - \ D^{[2]}_2 \ - \ \dots \ - \ D^{[2]}_N \ - \ D^{[1]}_{N+1}.
\end{align*}
We call the position perturbation parameter $\s>0$ and we apply the following procedure. In the material $A$, we increase the size of the particle $D^{[1]}_{N+1}$ by $\s$ and we decrease the size of the particle $D^{[1]}_{\frac{\lfloor N+2 \rfloor}{2}}$ by $\s$. Then, in a similar way, in the material $B$ we decrease the size of particle $D^{[1]}_1$ by $\s$ and we increase the size of $D^{[1]}_{\frac{\lfloor N+2 \rfloor}{2}}$ by $\s$. Hence, we manage to break the mirror symmetry governing each material. We will view this more analytically later in Figure \ref{fig:materials_A_B:pos_pert}, where we have considered a specific example for the model.

Let us note here that the position perturbation $\s$ cannot take arbitrarily big values, as the size of the periodic cell has to remain unaffected. In fact, for this to happen, we require that
\begin{align}\label{eq:sigma:upper_bound}
    \s \leq  \min_{A,B} \Big\{ |D^{[1]}_1|, \Big|D^{[1]}_{\frac{\lfloor N+2 \rfloor}{2}}\Big| \Big\}, 
\end{align}
where $\min_{A,B}$ denotes the minimum over materials $A$ and $B$, and $|\cdot|$ denotes the size of each particle. Since we work in the one-dimensional setting, $|\cdot|$ denoted the length.

Once again, as in Section \ref{subsubsection:theory:perm_pert}, using the transfer matrix method, we can see the continuous dependence that the decay mode has on the length of each particle. Thus, as a direct result, the following Lemma holds.

\begin{lemma}\label{lemma:IM_cvg:pos_pert}
    Let $u$ denote the asymptotically decaying interface mode of \eqref{eq:Helmholtz} in the band gap $\mathfrak{A}$ and let $u_{\s}$ denote an asymptotically decaying interface mode associated to the position perturbation procedure described above in the band gap $\mathfrak{A}_{\s}$. Then, if $\mathfrak{A} \cap \mathfrak{A}_{\s} \ne \emptyset$, it holds that
    \begin{align*}
        \lim_{\s\to0} u_{\s} = u.
    \end{align*}
\end{lemma}

\subsection{Numerical results}

Let us now provide an example, in order to solidify the previous analysis. We consider dispersive particles with permittivity given by
\begin{align}\label{def:ex:permittivity}
    \ve(x,\w) = 
    \begin{cases}
        \ve_1(\w), & x \in D^{[1]},\\
        \ve_2(\w), & x \in D^{[2]}, 
    \end{cases} 
    \ \ \text{ where } \ \ 
    \ve_i(\w) := \ve_0 + \frac{\a_i}{1-\b_i\w^2}, \ \ \text{ for } \ i=1,2,
\end{align}
with $\ve_0,\a_i,\b_i\geq0$, for $i=1,2$. This permittivity function model is inspired from the examples treated in \cite{alexopoulos2023effect} and represents the undamped case of a Drude material or a halide perovskite. 
In order to obtain analytic results we consider the model for the periodic cells of materials $A$ and $B$ described in Figure \ref{fig:materials_AandB}. We note that, in this case, each periodic cell is constituted by five particles, three with permittivity $\ve_1(\w)$ and two with permittivity $\ve_2(\w)$. The particles have the following ordering:
\begin{align*}
    D^{[1]}_1 \ - \ D^{[2]}_1 \ - \ D^{[1]}_2 \ - \ D^{[2]}_2 \ - \ D^{[1]}_3.
\end{align*}
Let us denote by $\w_{i,p}^{(\pm)}$ the poles of the permittivity \eqref{def:permittivity}, i.e.
\begin{align}\label{def:poles}
    \w_{i,p}^{(\pm)} = \pm \frac{\sqrt{\b_i}}{\b_i}.
\end{align}
We see that
\begin{align*}
    \frac{\upd\ve}{\upd\w}(x,\w) = \begin{cases}
        \frac{\upd\ve_1}{\upd\w}(\w), \ \ &x \in D^{[1]},\\
        \frac{\upd\ve_2}{\upd\w}(\w), \ \ &x \in D^{[2]}.
    \end{cases}
\end{align*}
Thus, it holds 
\begin{align*}
    \frac{\upd \ve_1}{\upd\w}(\w) &= \frac{\upd}{\upd\w} \left( \ve_0 + \frac{\a_1}{1-\b_1\w^2} \right) = \frac{2\a_1\b_1\w}{(1-\b_1\w^2)^2} \geq 0 
\end{align*}
and similarly
\begin{align*}
    \frac{\upd \ve_2}{\upd\w}(\w) \geq 0,
\end{align*}
which gives
\begin{align*}
    \frac{\upd\ve}{\upd\w}(x,\w) \geq 0, \ \ x\in\mathbb{R}.
\end{align*}
Hence, Theorem \ref{thm:bulksum} holds. This implies that if we consider a one-dimensional photonic crystal of the form of Figure \ref{fig:materials_AandB} with permittivity given by \eqref{def:ex:permittivity}, then, in each band-gap, a topologically protected interface mode exists.


\subsubsection{Permittivity perturbation}

As mentioned previously, we will first study perturbations on the material parameters while keeping the symmetry inside each periodic cell intact. This occurs by introducing a perturbation to the permittivity of the particles and it gives rise to the perturbed system described by \eqref{eq:Helmholtz:pert}.

We can see that if the perturbation function $f$ is constant, then the permittivity is not strongly affected. Hence we will consider perturbation functions which are dispersive with respect to the frequency $\w$. In order to simplify the analysis we will take $f_1(\w)=f_2(\w)$, for all $\w\in\mathbb{R}$. 

We will study the perturbation effect for two separate cases of perturbation functions:
\begin{enumerate}
    \item Perturbation functions which satisfy the assumptions of the original permittivity. In particular, we will consider the following:
    \begin{align} \label{def:perturbationfunction:resp}
        f(\w) = -\frac{1}{\w^2}.
    \end{align}
    \item Perturbation functions which does not satisfy the assumptions of the original permittivity. In particular, we will consider the following:
    \begin{align} \label{def:perturbationfunction:disresp}
        f(\w) = \frac{1}{\w^2}.
    \end{align}
\end{enumerate}

For both cases, we see directly the smoothness of the perturbation function. We will also work inside the first band gap of the problem \eqref{eq:Helmholtz:pert}. This region is away from zero, so we ensure that $\frac{\partial f}{\partial\w}<+\infty$.

For the first case, we observe that as the perturbation parameter $\d$ increases, the size of the band gap increases. Because of the monotonicity of the perturbation function $\f$, Theorem \eqref{thm:bulksum} holds and so a unique interface mode exists. This fact is shown in Figure \ref{fig:permittivity_perturbation}, where we plot for each value of $\d$ the value of the frequency $\w$ of the interface mode. We observe that as $\d$ increases, the interface mode frequency increases as well, which is natural, since the size of the band gap increases. We also note that, as $\d\to0$, the interface mode frequency converges to the one of the unperturbed interface mode, as shown in Lemma \ref{lemma:IM_cvg:perm_pert}.

\begin{figure}
\begin{center}
\begin{tikzpicture}
        \node[inner sep=0pt] (russell) at (4.2,6.5) {\includegraphics[width=0.7\textwidth]{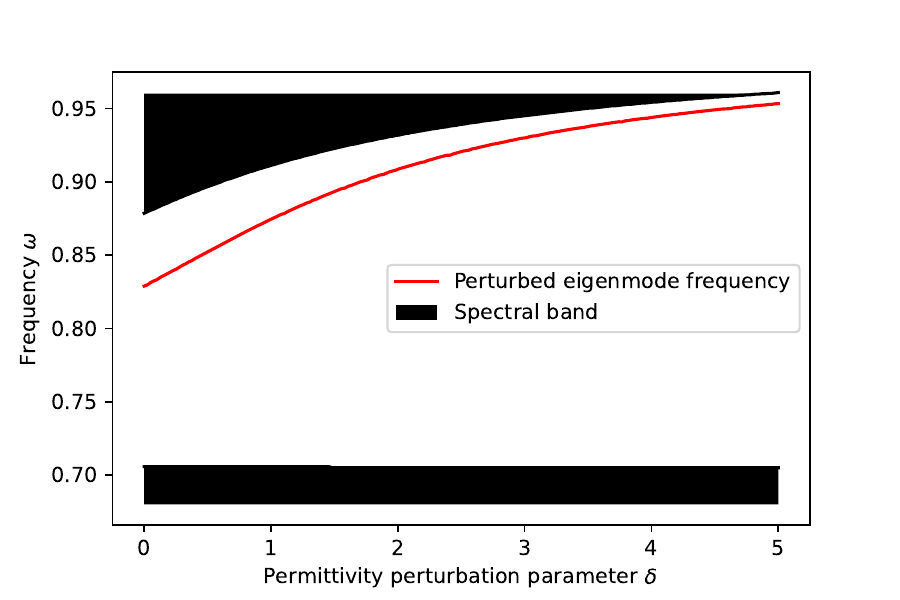}};
    \end{tikzpicture}
\end{center}
\caption{Plot of the frequencies of the interface mode and the edges of the band gap as perturbations with magnitude $\d$ that preserve the monotonicity condition are added. In particular, $\d$ times the strictly increasing function \eqref{def:perturbationfunction:resp} is added to the permittivity \eqref{def:ex:permittivity}. This function is in full accordance with the initial assumptions on the material permittivities. We see that as $\d\to0$, the perturbed interface mode frequency converges to the unperturbed one, as stated in Lemma \ref{lemma:IM_cvg:perm_pert}. As $\d$ increases, the size of the band gap increases and the interface mode frequency increases but still remains inside the band gap.} 
\label{fig:permittivity_perturbation}
\end{figure}

Let us now move to the second case. It holds that as $\d$ increases, the size of the band gap decreases. This happens because the perturbation function $f$, which is decreasing, becomes the dominating factor of the perturbed permittivity $\widetilde{\ve}$. In Figure \ref{fig:permittivity_perturbation_break}, we observe that as $\d$ increases, the interface mode frequency $\w$ decreases. In particular, for large values of $\d$, the band gap becomes very small. On the other hand, as $\d\to0$, the original permittivity is still the dominating factor of the perturbed permittivity and so, from Theorem \ref{thm:bulksum}, the interface mode still exists and converges to the unperturbed one.

This closing of the band gap poses problems for the physical implementation of these systems. In particular, since the edge mode approaches the edge of the band gap, its eigenfrequency will be close to that of a propagating Bloch mode. As a result, in any physical device, it is likely than any attempt to use this localised mode for wave guiding will fail as the nearby propagating mode will also be excited.


\begin{figure}
\begin{center}
\begin{tikzpicture}
        \node[inner sep=0pt] (russell) at (-4.2,6.5) {\includegraphics[width=0.7\textwidth]{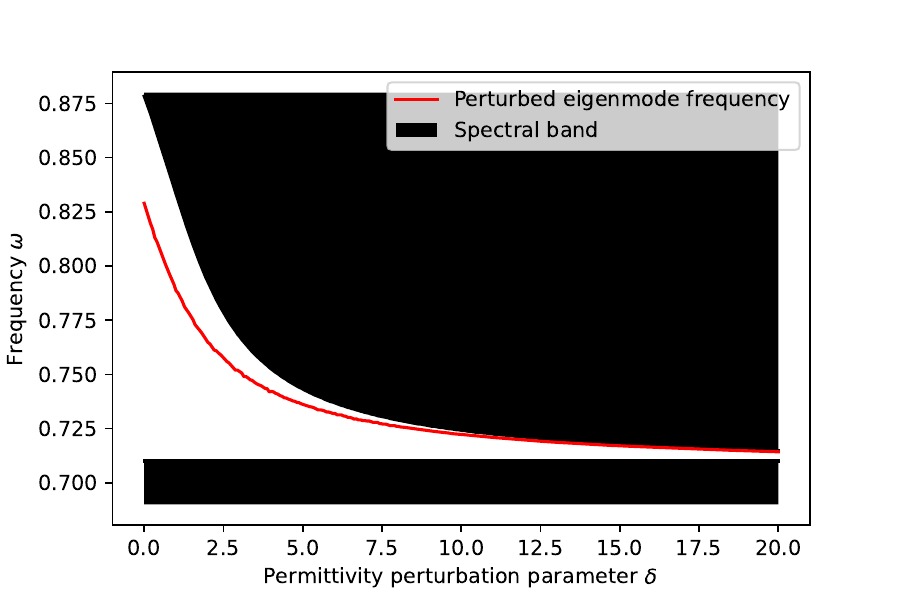}};
    \end{tikzpicture}
\end{center}
\caption{Plot of the frequencies of the interface mode and the edges of the band gap as perturbations with magnitude $\d$ that do not preserve the monotonicity condition are added. In particular, $\d$ times the strictly decreasing function \eqref{def:perturbationfunction:disresp} is added to the permittivity \eqref{def:ex:permittivity}. As a result, when $\d$ is sufficiently large, the permittivity function does not respect the monotonicity assumption. Of course, we see that as $\d\to0$, the perturbed interface mode frequency converges to the unperturbed one, as stated in Lemma \ref{lemma:IM_cvg:perm_pert}. As $\d$ increases, the width of the band gap decreases. As a result, the perturbed interface mode frequency gets pushed outside of the band gap. Hence, a localised interface mode no longer exists for frequency $\w$ in this band gap. The shaded areas are the spectral bands and the red line represents the frequency of the perturbed interface mode as the perturbation parameter $\d$ increases. We observe that the band gap becomes small for large values of $\d$ and that the perturbed interface mode frequency is pushed very close to the edge of the gap.
} 
\label{fig:permittivity_perturbation_break}
\end{figure}

\subsubsection{Symmetry perturbation}

Let us now treat the case where the symmetry of the periodic cells breaks. We will apply the procedure described in Section \ref{subsubsection:theory:pos_pert} on the model described in Figure \ref{fig:materials_AandB}. In particular, for material $A$, we reduce the size of the particle $D^{[1]}_2$ by $\s$ and we increase the size of particle $D^{[1]}_3$ by $\s$. Similarly, for material $B$, we decrease the size of $D^{[1]}_1$ by $\s$ and we increase the size of $D^{[1]}_2$ by the same amount. We depict this position perturbation graphically for the periodic cell of material $A$ and $B$ in Figure \ref{fig:materials_A_B:pos_pert}.

As it is already mentioned, the position perturbation parameter $\s$ cannot take arbitrarily large values, since the length of the periodic cell has to remain the same. Thus, from the bound \eqref{eq:sigma:upper_bound}, we see that
\begin{align*}
    0\leq\s\leq\th_2.
\end{align*}

\begin{figure}
    \centering
    \begin{tikzpicture}
        \node[inner sep=0pt] (russell) at (-4.5,6.5) {\includegraphics[width=0.47\textwidth]{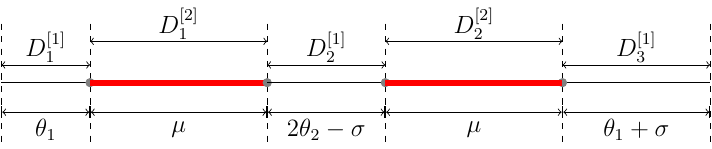}};
        \node at (-4.5,5.25) {\scriptsize Perturbed material A};
        \node[inner sep=0pt] (russell) at (4.5,6.5) {\includegraphics[width=0.47\textwidth]{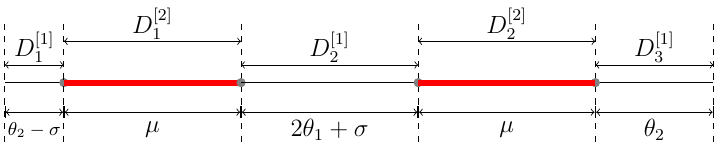}};
        \node at (4.5,5.25) {\scriptsize Perturbed material B};
    \end{tikzpicture}
    \caption{Position perturbation on the periodic cells of each material. We have considered as values $\th_1=0.1$, $\th_2=0.15$ and $\m=0.25$. We notice that the largest possible value for $\s$ is $0.15$. We see that by modifying the sizes of the particles, the mirror symmetry inside the unit cell no longer holds.} 
    \label{fig:materials_A_B:pos_pert}
\end{figure}

Studying the existence of an interface mode under the position perturbation regime is not in accordance with the argument we have presented in Theorem \ref{thm:bulksum}. In fact, our argument is based on the mirror symmetry governing the periodic cells and proposes a sufficient but not necessary condition for an interface mode to exist in a band gap. This symmetry breaks once the position perturbation takes place. What we can say with certainty is that, if the interface mode exists, then, as the position perturbation parameter $\s\to0$, it converges to the unperturbed interface mode (Lemma \ref{lemma:IM_cvg:pos_pert}). This result is shown in Figure \ref{fig:position_perturbation_break}, where we have plotted the interface mode frequency $\w$ as a function of the position perturbation parameter $\s$ inside a band gap.

\begin{figure}
\begin{center}
\begin{tikzpicture}
        \node[inner sep=0pt] (russell) at (4.2,6.5) {\includegraphics[width=0.7\textwidth]{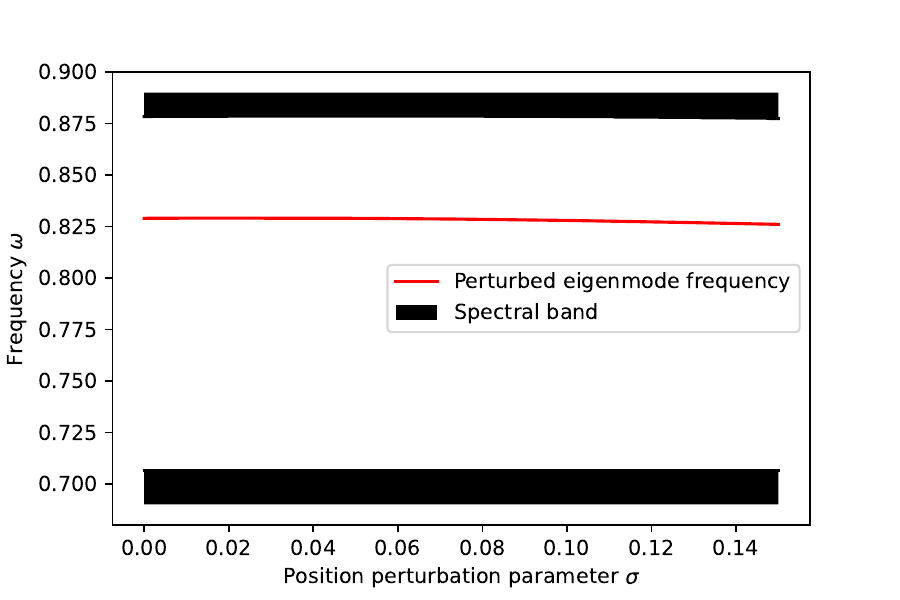}};
    \end{tikzpicture}
\end{center}
\caption{Plot of the interface mode frequency $\w$ as the position perturbation parameter $\s$ increases. We notice that as $\s\to0$, the perturbed interface mode frequency converges to the one of the unperturbed system. Also, we see that as $\s$ increases, the break of the symmetry inside the unit cell is more drastic and so, the perturbed interface mode frequency diverges from the unperturbed one. Although, this divergence is quite small and the perturbed interface mode frequency remains inside the band gap. We note that the shaded areas denote the spectral bands.} 
\label{fig:position_perturbation_break}
\end{figure}

\section{Conclusion}

Studying a one-dimensional system for dispersive materials, we used the mirror symmetry and the underlying periodic structure of our setting to show that localised interface modes exist. We explained the behaviour of the Zak phase, a topological invariant of the system, and we have provided the condition that the interface mode frequency needs to satisfy for these modes to decay at infinity. Finally, we introduced imperfections to the system in the form of perturbations on the material permittivity and the symmetry of the periodic cells and we showed how the decaying modes and their associated frequencies are affected.


\section*{Conflicts of interest}

The authors have no conflicts of interest to disclose.

\section*{Acknowledgements}

The work of KA was supported by ETH Z\"urich under the project ETH-34 20-2. The work of BD was funded by the Engineering and Physical Sciences Research Council through a fellowship with grant number EP/X027422/1.

\appendix

\section{Proofs of impedance monotonicity} \label{app:monotonicity}

\subsection{Proof of Theorem \ref{thm:Z:decrease}}

\begin{proof}[Proof of Theorem~\ref{thm:Z:decrease}]
    
We assume that we are working in a band gap $\mathfrak{A}$. We apply $\frac{\upd}{\upd \w}$ on both sides of \eqref{eq:Helmholtz} and we get, using \eqref{def:E,phi}, that
\begin{align*}
    \frac{\upd}{\upd \w} \left[ \frac{\upd}{\upd x} \left( E(x,\w) \frac{\upd u}{\upd x} \right) \right] = \frac{\upd}{\upd\w}\Big( - \m_0 \w^2 u \Big),
\end{align*}
which is
\begin{align*}
    \frac{\upd}{\upd x} \left[ \frac{\upd}{\upd \w} \left( E(x,\w) \frac{\upd u}{\upd x} \right) \right] = - 2 \m_0 \w u - \m_0 \w^2 \phi
\end{align*}
and so
\begin{align}\label{eq:Helmholtz:deriv}
    \frac{\upd}{\upd x} \left[ \frac{\upd E}{\upd \w} (x,\w) u' + E(x,\w) \f' \right] = - 2 \m_0 \w u - \m_0 \w^2 \phi.
\end{align}
Also, we have
\begin{align*}
    \frac{\upd Z^+}{\upd\w} = \frac{\upd}{\upd\w} \left( \frac{u}{E u} \right) = \frac{1}{(Eu')^2} \left[ E u' \f - u \left( \frac{\upd E}{\upd\w} u' + E \f' \right) \right].
\end{align*}
Let us define the Wronskian $\mathcal{W}$ by
\begin{align}
    \mathcal{W} := E u' \f - u \left( \frac{\upd E}{\upd\w} u' + E \f' \right).
\end{align}
Then, we get
\begin{align*}
    \frac{\upd Z_R}{\upd\w} = \frac{\mathcal{W}}{(Eu')^2}.
\end{align*}
We observe that
\begin{align*}
    \mathcal{W}' &= \frac{\upd W}{\upd x} = \frac{\upd}{\upd x} \left[  E u' \f - u \left( \frac{\upd E}{\upd\w} u' + E \f' \right) \right] \\
    &= E' u' \f + E (u''\f + u'\f') - u' \left( \frac{\upd E}{\upd\w} u' + E \f' \right) - u \left( \frac{\upd E}{\upd\w} u' + E \f' \right)'\\
    &= E' u' \f + E u''\f + E u'\f' - (u')^2 \frac{\upd E}{\upd\w} - E u' \f' - u \left( \frac{\upd E}{\upd\w} u' + E \f' \right)'\\
    &= E' u' \f + E u''\f - (u')^2 \frac{\upd E}{\upd\w} - u \left( \frac{\upd E}{\upd\w} u' + E \f' \right)'\\
    &= \f (Eu')' - u \left( \frac{\upd E}{\upd\w} u' + E \f' \right)' - (u')^2 \frac{\upd E}{\upd\w}.
\end{align*}
Applying \eqref{eq:Helmholtz} and \eqref{eq:Helmholtz:deriv}, we get
\begin{align*}
    \mathcal{W}' &= - \f \m_0 \w^2 u - u ( - 2 \w \m_0 u - \m_0 \w^2 \f ) - (u')^2 \frac{\upd E}{\upd\w}\\
    &= 2 \m_0 \w u^2- (u')^2 \frac{\upd E}{\upd\w},
\end{align*}
which gives
\begin{align}
    \mathcal{W}' = 2 \m_0 \w u^2- (u')^2 \frac{\upd E}{\upd\w}.
\end{align}
Integrating with respect to $x$ over the material $B$, i.e. on $[0,+\infty]$, we have
\begin{align*}
    \lim_{x\to+\infty} \mathcal{W}(x,w) - \mathcal{W}(0,w) = 2 \m_0 \w \int_0^{+\infty} u^2 \upd x - \int_0^{+\infty} (u')^2 \frac{\upd E}{\upd\w} \upd x.
\end{align*}
We know that when the frequency is inside a band gap, the field vanishes as $x\to+\infty$. This implies that
\begin{align*}
    \lim_{x\to+\infty} \mathcal{W}(x,w) = 0,
\end{align*}
and so
\begin{align*}
     \mathcal{W}(0,w) = \int_0^{+\infty} (u')^2 \frac{\upd E}{\upd\w} \upd x - 2 \m_0 \w \int_0^{+\infty} u^2 \upd x.
\end{align*}
We know that $\frac{\upd\ve}{\upd\w}\geq0$. This gives
\begin{align*}
    \frac{\upd E}{\upd\w} = \frac{\upd}{\upd\w}\left(\frac{1}{\ve}\right) = \frac{-1}{\ve^2} \frac{\upd\ve}{\upd\w} \leq 0.
\end{align*}
Hence,
\begin{align*}
    \int_0^{+\infty} (u')^2 \frac{\upd E}{\upd\w} \upd x \leq 0.
\end{align*}
Therefore, we get
\begin{align*}
    \frac{\upd Z^+}{\upd \w} = \frac{1}{(Eu')^2}\left(-2\m_0\w\int_0^{+\infty} u^2 \upd x + \int_{0}^{+\infty} (u')^2 \frac{\upd E}{\upd\w} \upd x \right)<0.
\end{align*}
We apply the same argument for $Z^-$ and we obtain similarly that
\begin{align*}
    \frac{\upd Z^-}{\upd \w}  < 0.
\end{align*}
This concludes the proof.

\end{proof}

\subsection{Proof of Theorem \ref{thm:pertZ:decrease}}

\begin{proof}[Proof of Theorem \ref{thm:pertZ:decrease}]
    The proof follows the same reasoning as the one of Lemma \ref{lemma:impedance}. We first define $\widetilde{E}$ by
    \begin{align*}
        \widetilde{E}(x,\w) := \frac{1}{\widetilde{\ve}(x,\w)}.
    \end{align*}
    Applying $\frac{\upd}{\upd \w}$ on both sides of $\eqref{eq:Helmholtz:pert}$, we get
    \begin{align}\label{eq:Helmholtz:pert:deriv}
        \frac{\upd}{\upd x} \left[ \frac{\upd \widetilde{E}}{\upd \w} (x,\w) u' + \widetilde{E}(x,\w) \f' \right] = - 2 \m_0 \w u - \m_0 \w^2 \phi,
    \end{align}
    where $\phi:=\frac{\upd u}{\upd\w}$. Also, we have
    \begin{align*}
        \frac{\upd \widetilde{Z}^+}{\upd\w} = \frac{1}{(\widetilde{E}u')^2} \left[ \widetilde{E} u' \f - u \left( \frac{\upd \widetilde{E}}{\upd\w} u' + \widetilde{E} \f' \right) \right].
    \end{align*}
    We define the perturbed Wronskian $\widetilde{\mathcal{W}}$ by
    \begin{align}
        \widetilde{\mathcal{W}} := \widetilde{E} u' \f - u \left( \frac{\upd \widetilde{E}}{\upd\w} u' + \widetilde{E} \f' \right),
    \end{align}
    which gives
    \begin{align*}
        \frac{\upd \widetilde{Z}^+}{\upd\w} = \frac{\widetilde{\mathcal{W}}}{(\widetilde{E}u')^2}.
    \end{align*}
    It follows that
    \begin{align*}
        \widetilde{\mathcal{W}}' = \f (\widetilde{E}u')' - u \left( \frac{\upd \widetilde{E}}{\upd\w} u' + \widetilde{E} \f' \right)' - (u')^2 \frac{\upd \widetilde{E}}{\upd\w},
    \end{align*}
    and so, from \eqref{eq:Helmholtz:pert} and \eqref{eq:Helmholtz:pert:deriv}, we get
    \begin{align*}
        \widetilde{\mathcal{W}}' = 2 \m_0 \w u^2- (u')^2 \frac{\upd \widetilde{E}}{\upd\w}.
    \end{align*}
    Integrating with respect to $x$ over the perturbed material $B$, i.e. on $[0,+\infty]$, we have
    \begin{align*}
        \lim_{x\to+\infty} \widetilde{\mathcal{W}}(x,w) - \widetilde{\mathcal{W}}(0,w) = 2 \m_0 \w \int_0^{+\infty} u^2 \upd x - \int_0^{+\infty} (u')^2 \frac{\upd \widetilde{E}}{\upd\w} \upd x.
    \end{align*}
    We know that when in a band gap, the field vanishes as $x\to+\infty$. This implies that
    \begin{align*}
        \lim_{x\to+\infty} \widetilde{\mathcal{W}}(x,w) = 0,
    \end{align*}
    and so
    \begin{align*}
        \widetilde{\mathcal{W}}(0,w) = \int_0^{+\infty} (u')^2 \frac{\upd \widetilde{E}}{\upd\w} \upd x - 2 \m_0 \w \int_0^{+\infty} u^2 \upd x.
    \end{align*}
    We observe that
    \begin{align*}
        \frac{\upd\widetilde{E}}{\upd\w} = \frac{\upd}{\upd\w} \left(\frac{1}{\widetilde{\ve}(x,\w)}\right) = \frac{\upd}{\upd\w} \left(\frac{1}{\ve(x,\w)+\d f(\w)}\right) = \frac{-1}{\Big(\ve(x,\w)+\d f(\w)\Big)^2} \left( \frac{\upd\ve}{\upd\w} + \d \frac{\upd f}{\upd\w} \right).
    \end{align*}
    Thus, since $\frac{\partial f}{\partial \w}<+\infty$, we get
    \begin{align*}
        \lim_{\d\to0} \frac{\upd\widetilde{E}}{\upd\w} = \frac{-1}{\ve(x,\w)^2} \frac{\upd\ve}{\upd\w} \leq 0,
    \end{align*}
    since we have that $\frac{\upd\ve}{\upd\w}\geq0$. Hence,
    \begin{align*}
        \lim_{\d\to0} \int_0^{+\infty} (u')^2 \frac{\upd \widetilde{E}}{\upd\w} \upd x \leq 0,
    \end{align*}
    Therefore, at $x=0$, we get
    \begin{align*}
        \lim_{\d\to0}\frac{\upd \widetilde{Z}^+}{\upd \w} = \lim_{\d\to0} \frac{1}{(\widetilde{E}u')^2}\left(-2\m_0\w\int_0^{+\infty} u^2 \upd x + \int_{0}^{+\infty} (u')^2 \frac{\upd \widetilde{E}}{\upd\w} \upd x \right)<0.
    \end{align*}
    The same argument holds for $\widetilde{Z}^-$ and so, we obtain that
    \begin{align*}
        \frac{\upd \widetilde{Z}^-}{\upd \w}  < 0.
    \end{align*}
    This concludes the proof.
\end{proof}

\section{Characterisation of eigenmode symmetries} \label{app:symmetries}

\begin{proof}[Proof of Lemma~\ref{lemma:u,u'}]
    The first thing that we notice is that we cannot have $u_n^{[0]}(0) = \Big(u_n^{[0]}\Big)'(0) = 0$ or $u_n^{[\pi]}(0) = \Big(u_n^{[\pi]}\Big)'(0) = 0$, since by the uniqueness of solution to the problem \eqref{eq:Helmholtz}, it would mean $u_n\equiv0$. Then, from the Lemma \ref{lemma:modesymmetry}, we have two different cases. We will treat each one separately:
    \begin{itemize}
        \item We have $u_n$ being symmetric and $u'_n$ being anti-symmetric. At $\k=\pi$, from the symmetry of $u_n$, we have that
        \begin{align*}
            u_n^{[\pi]}(0) = \mathcal{P}u_n^{[\pi]}(0) = u_n^{[\pi]}(1)
        \end{align*}
        and from \eqref{bd1}, it holds that
        \begin{align*}
            u_n^{[\pi]}(0) = e^{i\pi} u_n^{[\pi]}(1) = - u_n^{[\pi]}(1). 
        \end{align*}
        Combining the two, we get
        \begin{align*}
            u_n^{[\pi]}(0) = 0.
        \end{align*}
        Then, at $\k=0$, since $u'_n$ is anti-symmetric, we have that
        \begin{align*}
            \Big(u_n^{[0]}\Big)'(0) = -\mathcal{P}\Big(u_n^{[0]}\Big)'(0) = - \Big(u_n^{[0]}\Big)'(1).
        \end{align*}
        But, from \eqref{bd1}, we have
        \begin{align*}
            \Big(u_n^{[0]}\Big)'(0) = e^0 \Big(u_n^{[0]}\Big)'(1) = \Big(u_n^{[0]}\Big)'(1).
        \end{align*}
        Combining the two, we get
        \begin{align*}
            \Big(u_n^{[0]}\Big)'(0) = 0.
        \end{align*}
        This gives the desired result.
        \item The second case is $u_n$ being anti-symmetric and $u'_n$ being symmetric. At $\k=0$, the anti-symmetry of $u_n$ gives
        \begin{align*}
            u_n^{[0]}(0) = - \mathcal{P}u_n^{[0]}(0) = -u_n^{[0]}(1)
        \end{align*}
        and the quasiperiodic boundary condition \eqref{bd1} gives
        \begin{align*}
            u_n^{[0]}(0) = e^{0} u_n^{[0]}(1) = u_n^{[0]}(1). 
        \end{align*}
        Combining the two, we get
        \begin{align*}
            u_n^{[0]}(0) = 0.
        \end{align*}
        Then, at $\k=\pi$, since $u'_n$ is symmetric, we have that
        \begin{align*}
            \Big(u_n^{[\pi]}\Big)'(0) = \mathcal{P}\Big(u_n^{[\pi]}\Big)'(0) =  \Big(u_n^{[\pi]}\Big)'(1).
        \end{align*}
        But, from \eqref{bd1}, we have
        \begin{align*}
            \Big(u_n^{[\pi]}\Big)'(0) = e^{i\pi} \Big(u_n^{[\pi]}\Big)'(1) = - \Big(u_n^{[\pi]}\Big)'(1).
        \end{align*}
        Combining the two, we get
        \begin{align*}
            \Big(u_n^{[\pi]}\Big)'(0) = 0.
        \end{align*}
        This gives the desired result.
    \end{itemize}
    This concludes the proof.
\end{proof}

\bibliographystyle{abbrv}
\bibliography{references}{}

\end{document}